\documentclass[a4paper,10pt]{article}
\usepackage{mathtools,amsfonts,amssymb,amsthm,mathrsfs}
\usepackage{lmodern}
\usepackage{graphicx}
\usepackage{subcaption}
\usepackage{tikz}
\usepackage{pstricks}
\usepackage{tabularx}
\usepackage[margin=3cm]{geometry}
\usepackage{enumitem}
\usepackage[english]{babel}
\usepackage[T1]{fontenc}
\usepackage[utf8]{inputenc}
\usepackage[notref,notcite]{showkeys} % Package to see the keys of labelling
\usepackage{comment}
\usepackage{authblk} % Package to include several authors
\usepackage[normalem]{ulem}
\usepackage{dsfont} % Package for the characteristic function

\usepackage{chngcntr} % Package to reset the equation numbering at each subsection

\nocite{*}

\counterwithin*{equation}{section}

\theoremstyle{plain}

\newtheorem{theorem}{Theorem}[section]
\newtheorem{lemma}[theorem]{Lemma}
\newtheorem{proposition}[theorem]{Proposition}
\newtheorem{corollary}[theorem]{Corollary}

\theoremstyle{definition}

\newtheorem{assumption}{Assumption}[section]

\usepackage{hyperref}
\hypersetup{colorlinks,%
            citecolor=red,%
            filecolor=yellow,%
            linkcolor=black,%
            urlcolor=black}

\newcommand{\ovbeta}{\overline{\beta}}
\newcommand{\unmu}{\underline{\mu}}

\newcommand \ph {\varphi}

\newcommand \N {\mathbb N}

\newcommand \R {\mathbb R}

\renewcommand{\L}{\mathrm{L}} % For the L^p spaces
\newcommand{\W}{\mathrm{W}}

\renewcommand \AA {\mathcal{A}} % La commande \AA existe deja de base

\newcommand \LL {\mathcal{L}}  % La commande \L existe deja de base

\newcommand{\inte}{\operatorname{int}}

\newcommand{\1}{\mathds{1}}

\DeclareMathOperator{\esup}{ess\,sup}
\DeclareMathOperator{\supp}{supp}
\newcommand{\dd}{\mathrm{d}}
\newcommand{\df}[2]{\frac{\dd#1}{\dd#2}}
\renewcommand{\Re}{\mathrm{Re}}
\newcommand{\vcol}[1]{\left( \begin{smallmatrix} #1 \end{smallmatrix} \right)}

\title{Sharp threshold dynamics for a bistable age-structured population model}
\author[1]{Quentin Griette}
\author[1,2]{Franco Herrera\footnote{Corresponding author}}
\affil[1]{ LMAH, Université Le Havre Normandie, 76600, Le Havre, France; \href{mailto:quentin.griette@univ-lehavre.fr}{quentin.griette@univ-lehavre.fr}; \href{mailto:franco.herrera-granda@etu.univ-lehavre.fr}{franco.herrera-granda@etu.univ-lehavre.fr}}
\affil[2]{Instituto de Matemáticas, Universidad de Talca, Talca, Chile; \href{mailto:franco.herrera@utalca.cl}{franco.herrera@utalca.cl}}

\begin{document}

\maketitle

\begin{abstract}
This paper is devoted to the long-term dynamics of solutions to the Gurtin-MacCamy population model with a bistable birth function. We consider a one-parameter monotone family of initial distributions for the population such that for small values of the parameter, the corresponding population density gets extinct as time passes, whereas for large values of them, the solutions exhibit a different behavior. We are interested in the intermediate set of values for the parameters, which are called threshold parameters. We prove the existence of a sharp transition between these two asymptotic dynamics; that is, there exists exactly one threshold value when the age-dependent birth rate of the population has compact support, utilizing the theory of monotone dynamical systems. The case when the birth rate is non-compactly supported is more intricate to deal with, as has been observed in several works, even if the nonlinear birth function is monostable. Nevertheless, the approach used in the present work turns out to be effective to handle a particular birth rate with noncompact support by translating the dynamics of the age-structured model into an integro-differential system.

\vspace{0.2in}\noindent \textbf{Key words}. Age-structured model, sharp threshold transition, bistable birth function, delay-differential equation, monotone dynamical system.

\vspace{0.1in}\noindent \textbf{2020 Mathematical Subject Classification:} Primary: 35B40; Secondary: 35B41, 37N25, 45D05, 45M05.

\end{abstract}

\section{Introduction}\label{sec:intro}
This paper is devoted to the asymptotic behavior of the solutions to the following age-structured population model, which is usually referred to in the literature under the name of Gurtin-MacCamy model \cite{MR354068}:
\begin{equation}\label{GM-model}
    \begin{dcases}
    \partial_t u(t,a) + \partial_a u(t,a) = -\mu(a) u(t,a), & t>0,~a>0, \\
    u(t,0) = f\left( \int_0^{+\infty} \beta(a) u(t,a) \dd a \right), & t>0, \\
    u(0,\cdot) = u_0 \in \L^1_+(0,+\infty).
    \end{dcases}
\end{equation}
The differential equation appearing in \eqref{GM-model} is known as the McKendrick-von Foerster equation, and it is widely used in different areas of mathematical biology, such as demography, cell proliferation, and epidemiology \cite{MR3988617,MR2748390,MR3616174,MR4476225,MR4557625,MR2017740,Ruiz}. In this model, the function $u(t,a)$ denotes the density of a population at time $t$ of age $a$, $\mu$ is the age-specific mortality rate, $\beta$ is the age-specific fertility rate, $f$ is a non-negative birth function such that $f(0)=0$, and $u_0$ represents the initial distribution of the population. Including the function $f$ in the model is a way to consider the {\it effective} offspring produced by the individuals in order to study situations where the size of the population may have a positive or negative effect on the system, for instance, including a crowding effect or saturation for the environment. Its name is coined in honor of the pioneering work \cite{MR354068} proposing a nonlinear age-structured model to describe the time evolution of a theoretical population.

Various aspects of this model have been extensively studied, with a major focus on the existence and stability of the steady states. Mainly, the goal of those studies has been to establish general conditions assuring the existence of a unique nontrivial steady state and also to provide a criterion for its local or global stability. The use of functional analytic tools \cite{MR3616174,MR772205} in companion with the development of semigroups, integrated semigroups \cite{MR3887640}, and the theory of functional differential equations \cite{MR1345150} has provided an effective framework to face such age-dependent models.

Equation \eqref{GM-model} with the choice of a Ricker's-type nonlinearity $f(x) = \alpha xe^{-x}$, with $\alpha>0$, was studied by Magal and Ma \cite{MR4683856}, where they showed the global asymptotic stability of the unique nontrivial equilibrium whenever $\alpha\in (1,e^2)$ and proved complementarily that a Hopf bifurcation may occurs at $\alpha = \alpha^\star > e^2$. Herrera and Trofimchuk \cite{MR4750985} worked on an integral equation arising naturally from this model (see Section \ref{sec:prelim}), proving the global asymptotic stability of the unique nontrivial equilibrium under a quite generic monostable shape for the nonlinearity $f$, namely a unimodal one. It is also possible to include spatial structure in the model by adding a diffusion term. Ruiz-Herrera and Touaoula \cite{Ruiz} consider this situation and make a link between the age-structured model and the one-dimensional dynamical system generated by $f$, from where some properties, such as the global attraction of the steady states, can be deduced from the properties of the birth function. On the other hand, Kang and Ruan \cite{MR4476225} studied the principal spectral theory of an age-structured model with nonlocal diffusion for the case where $f$ is just the identity function. They stated some conditions guaranteeing the existence of the principal eigenvalue and ensuring that the semigroup generated by the solutions to the equation  exhibits asynchronous exponential growth. Later, in a series of two works, Ducrot et al. \cite{MR4905314,MR4905582} studied the same model with monotone nonlinearity on the birth function and including a nonlocal diffusion term. They provide criteria for the existence of principal eigenvalues and study their effects on the global dynamics of the model via the sign of the spectral bound of some linearized operator. See also \cite{MR4984351} for a study in a spatially periodic media.

Each of the aforementioned works deals, in the end, with a monostable situation for the purpose of describing the global dynamics of the model. In other words, the results concerning the asymptotic behavior of the nonzero solutions of \eqref{GM-model} are given when there exists a unique stationary solution. The study of the monostable case is usually addressed using the principal eigenvalue, generalized principal eigenvalue, and/or the spectral radius to define the so-called basic reproduction number and to determine a threshold value. This threshold value serves as a bifurcation point at which the trivial equilibrium loses stability, giving rise to a unique non-trivial equilibrium, generally locally stable. Thus, this threshold value is used to study for which range of parameter values the solutions of the model tend either to the trivial equilibrium or to the unique non-trivial equilibrium, as well as local or global stability results \cite{MR968836,MR4931087,MR4750985,MR4557625,MR2017740,MR4392478}.

In the present work, we wish to pay some attention to the bistable scenario. That is, we are interested in the dynamics of the solutions when the function $f$ has exactly three fixed points, one trivial and two nontrivial, with the intermediate one being unstable (see Figure \ref{fig-1}), and therefore the Gurtin-MacCamy model admits exactly three steady states, namely
\[
    \overline{\varphi}_0 = 0, \quad \overline{\varphi}_1(a) = \kappa_1 e^{-\int_0^a \mu(l) \dd l}, \quad \text{and}\quad \overline{\varphi}_2(a) = \kappa_2 e^{-\int_0^a \mu(l) \dd l}.
\]
Under these conditions, the model adopts a feature characteristic of the Allee effect observed in several biological populations and structured models \cite{MR1310123,MR2748390,MR4952628,MR2320699}. That is, the existence of a critical value such that the population cannot persist in time provided the initial population is below it. To give a biological interpretation of this case, we may identify separately the offspring of the individuals and the effective newborns. For instance, the loggerhead sea turtle ({\it Caretta caretta}) is a species that lays its eggs in the sand, and once these hatch, the loggerhead hatchlings have to reach the sea. During this vulnerable journey, plenty of predators stalk the newborns. It has been observed that if the number of hatchlings per emergence is small, then the survival chance is reduced, while in the opposite case the chance of surviving increases due to a predator-swamping effect \cite{erb2019nest}. Thus, the assumption of $f$ being bistable has an interesting biological meaning.

Assuming that $f$ has a bistable structure as described above, then we may expect two typical asymptotic behaviors for the solutions, namely the convergence towards one of the stable equilibria. In the case of reaction-diffusion equations with this type of reaction term, this question has been addressed previously for several classes of initial data; see \cite{MR2608941,MR2754337,MR2169048}. Substantially, the results in these works are concerned with a sharp transition in the behavior of a monotone family of solutions. In order to set up the ideas in our context, consider a family of nonnegative initial distributions $\{ u_\lambda \}_{\lambda \geq0}$ with $u_0 \equiv 0$ and such that the map $\lambda \mapsto u_\lambda$ is increasing and continuous in the $\L^1$-norm. We denote by $u^\lambda$ the solution to the equation \eqref{GM-model}. Due to the asymptotic stability of the trivial equilibrium, it is clear that for small values of $\lambda$, the solution verifies $u^\lambda(t,\cdot) \to 0$ as $t\to \infty$ in $\L^1(0,+\infty)$. In the same fashion, if for large values of $\lambda$, $u_\lambda$ becomes sufficiently large to keep above the nontrivial stable equilibrium (see Corollary \ref{cor:conv}), then $u^\lambda(t,\cdot) \to \overline{\varphi}_2$. Accordingly, one may be interested in the so-called threshold solutions, i.e., solutions that do not exhibit any of the aforementioned behavior, and determine whether there is no one, exactly one, or a continuum of them. From the viewpoint of applications, the existence of exactly one threshold value will indicate a sharp transition between extinction and stabilization, and any intermediate behavior will be very exceptional. On the other hand, from the mathematical viewpoint, the description of the threshold solutions is also a fascinating problem and usually way harder. Mainly, the instability of the intermediate equilibrium may produce a positive feedback cycle that gives rise to oscillatory behaviors.

In the hope of understanding how the threshold solutions behave for the age-structured model, we may establish a connection with the theory of delay differential equations. Let the fertility rate be a step function, $\beta(a) = \beta \1_{a\geq \tau}$ with $\beta>0$ and $\tau>0$, and for the sake of simplicity, let $\mu(a) \equiv \mu>0$ be a constant. Integrating the differential equation in \eqref{GM-model} with respect to $a$ we find
\[
    \partial_t \int_0^{+\infty} u(t,a) \dd a = f\left( \beta \int_\tau^{+\infty} u(t,a) \dd a \right) - \mu \int_0^{+\infty} u(t,a) \dd a.
\]
Alternatively, integrating the same equation along the characteristics in the interval $[t-\tau, t]$ yields, for $t\geq \tau$ that
\[
    u(t,a) = \begin{cases}
        e^{-\mu \tau} u(t-\tau, a-\tau), & \text{if}~ a\geq \tau, \\
        e^{-\mu a} f\left( \beta \int_\tau^{+\infty} u(t-a, \sigma) \dd \sigma \right), & \text{if}~ a\leq \tau,
    \end{cases}
\]
so by changing variables we have
\[
    \int_\tau^{+\infty} u(t,a) \dd a = e^{-\mu \tau} \int_0^{+\infty} u(t-\tau,a) \dd a.
\]
Therefore, letting $U(t) := \int_0^{+\infty} u(t,a) \dd a$ we find that the system \eqref{GM-model} can be rewritten as a delay differential equation in $U$
\begin{equation}\label{eq:delay}
    U'(t) = f(\beta e^{-\mu \tau} U(t-\tau)) - \mu U(t), \quad t>\tau,
\end{equation}
equipped with an initial data $U(t)= U_0(t)$ for $t\in [0,\tau]$. Equation \eqref{eq:delay} is also known in the literature as a Mackey-Glass type equation, and the structure imposed over $f$ gives a positive feedback condition on the equation. Here we would like to refer to the remarkable reference \cite{MR1719128}, where the authors give a quite general description of the dynamics for this equation, which actually has a whole branch of applications in neural network theory \cite{MR1834537}. In the cited reference, jointly with \cite{MR1822211}, the authors proved that the global attractor of the system can be visualized as a 3-dimensional solid spindle with the stable equilibria as tips, and it is split by an invariant disk into the basins of attraction toward each of these tips. These works motivate various ideas used in the present one.

\subsection{Assumptions and main results}

We will use the following notation all over the work for the set of non-negative real numbers: $\R_+ = [0,+\infty)$. Moreover, if $X$ denotes a Banach space and $X_+$ is a positive cone of $X$, we will use the following notation for the different order relations induced by $X_+$ on the elements of $X$. Let $x,y\in X$ and define the partial order relation on $X$ as follows
\begin{align*}
    x \leq y &\iff y- x \in X_+, \\
    x < y &\iff x\leq y ~\text{and}~ x\neq y, \\
    x \ll y &\iff y-x \in \inte X_+,
\end{align*}
where $\inte X_+$ denotes the interior of $X_+$. It is similarly defined $x\geq y$, $x>y$ and $x\gg y$.  In particular, the space of continuous functions over a compact interval is endowed with the order given by the cone of positive continuous functions.

This study is devoted to the bistable case for the nonlinearity $f$ in \eqref{GM-model}. Mainly, we will assume that the following statements hold throughout the paper.

\begin{assumption}\label{as-1}
The following conditions over the parameters in \eqref{GM-model} hold:
\begin{enumerate}[label = (\roman*)]
    \item $\beta,\mu \in \L^\infty_+ (0,+\infty)$ and are normalized in the following sense
    \begin{equation}\label{eq:norm}
        \int_0^{+\infty} \beta(a) e^{-\int_0^a \mu(l)\dd l}\dd a = 1.
    \end{equation}
    \item There exists $\underline{\mu}>0$ such that $\mu(a) \geq \underline{\mu}$ for almost every $a\geq 0$.
    \item $f\colon \R_+ \to \R_+$ is a $C^1$-function with bounded derivative such that $f'(x)>0$ for every $x\in \R_+$. It possesses exactly three equilibria, i.e., the equation $f(x)=x$ has the solution set $\{0,\kappa_1,\kappa_2\}$. Also, the relations $f'(0),f'(\kappa_2)<1$ and $f'(\kappa_1)>1$ are valid. See Figure \ref{fig-1}.
    \item Moreover, $f$ verifies
    \begin{equation}\label{diag_cond}
        \limsup_{x\to \infty} \frac{f(x)}{x} < 1.
    \end{equation}
\end{enumerate}
\end{assumption}

\begin{figure}[t]
    \centering
    \begin{tikzpicture}[scale=1, line width=0.7pt]
    \draw[->] (0,0) -- (4.5,0) node [anchor=west] {$x$};
    \draw[->] (0,0) -- (0,4.5) node [anchor=south] {$y$};
    \draw (0,0) -- (4.4,4.4) node [anchor=south] {$y=x$};
    \draw (0,0) .. controls (0.5,0.1) and (1,0.4) .. (1.5,1.5) .. controls (2,2.6) and (2.5,2.9) .. (3,3) .. controls (3.5,3.1) and (4,3.2) .. (4.3,3.3) node [anchor=west] {$y=f(x)$};
    \draw[dashed] (1.5,1.5) -- (1.5,0) node [anchor=north] {$\kappa_1$};
    \draw[dashed] (3,3) -- (3,0) node [anchor=north] {$\kappa_2$};
    \end{tikzpicture}
    \caption{Graph of the nonlinearity $f\colon \R_+ \to \R_+$.}
    \label{fig-1}
\end{figure}

Hereinafter we set $\ovbeta$ to be some positive number such that $\beta(a) \leq \ovbeta$ for almost every $a\geq 0$. Under this assumption, we have the following description about the steady states of the model \eqref{GM-model}.

\begin{proposition}
The Gurtin-MacCamy model possesses exactly three equilibria, say 0, $\overline{\varphi}_1$ and $\overline{\varphi}_2$, where the last are given by
\begin{equation}\label{def:equib}
    \overline{\varphi}_1(a) = \kappa_1 e^{-\int_0^a \mu(l) \dd l} \quad \text{and} \quad \overline{\varphi}_2(a) = \kappa_2 e^{-\int_0^a \mu(l) \dd l}.
\end{equation}
\end{proposition}
\begin{proof}
Let $\overline{\varphi}(a)$ be a steady state of the Gurtin-MacCamy model. Thus, the differential equation in \eqref{GM-model} turns into $\overline{\varphi}'(a) = -\mu(a) \overline{\varphi}(a)$, which yields $\overline{\varphi}(a) = \overline{\varphi}(0) e^{-\int_0^a \mu(l) \dd l}$. Using the boundary condition and the formula just obtained, we get that $\overline{\varphi}(0) = f\left( \overline{\varphi}(0) \int_0^{+\infty} \beta(a)e^{-\int_0^a \mu(l) \dd l}  \dd a \right)$. Thus, from the normalization hypothesis given in Assumption \ref{as-1}, we deduce that $\overline{\varphi}(0)\in \{0, \kappa_1, \kappa_2\}$, completing the proof.
\end{proof}

An important quantity to consider is the maximal reproductive age for the population. Mathematically, it can be defined as
\begin{equation}\label{def:a*}
    a^* := \sup\left\{ a\geq 0 \colon \int_a^{+\infty} \beta(a)e^{-\int_0^a \mu(l) \dd l} \dd a > 0  \right\} \in (0,+\infty].
\end{equation}
Observe that this definition implies that $\beta(a) = 0$ almost everywhere on $(a^*,+\infty)$ and
\[
    \int_\sigma^{a^*} \beta(a) e^{-\int_0^a \mu(l) \dd l} \dd a >0, \quad \forall \sigma\in[0,a^*).
\]
The case $a^*=+\infty$ is more intricate to treat due to the unboundedness in the domain, and it requires some additional assumptions in the fertility rate; see \cite{MR4905314,MR4905582}.

To state our main result, we also consider a family of initial data $\{ u_\lambda \}_{\lambda\geq 0} \subseteq \L^1_+(0,+\infty)$ under the following assumption.

\begin{assumption}\label{as-2}
The following conditions over the family $\{u_\lambda\}_{\lambda \geq 0}$ hold:
\begin{enumerate}[label=(\roman*)]
    \item The map $\lambda \to u_\lambda$ is continuous as a map from $\R_+$ into $\L^1(0,+\infty)$, with $u_0 \equiv 0$.
    \item The family $\{ u_\lambda \}_{\lambda \geq 0}$ is monotone increasing, in the sense that if $\lambda < \eta$, then $u_\lambda \leq u_\eta$ a.e and the inequality is strict in some subset of $(0,a^*)$ of positive measure.
\end{enumerate}
\end{assumption}

We denote by $u^\lambda$ the solution of \eqref{GM-model} with initial distribution $u_\lambda$. The last condition, due to the definition of $a^*$, can be interpreted as follows: each initial distribution $u_\eta$ for the initial population will create some extra mass than the preceding ones $u_\lambda$ with $\lambda < \eta$, as the number of individuals in reproductive age increases as the parameter does. Our main result is the following:

\begin{theorem}\label{th:main}
Suppose Assumptions \ref{as-1} and \ref{as-2} are valid, and that $a^*<+\infty$ holds. We have the following alternative:
\begin{itemize}
    \item Either for every $\lambda \geq 0$, $u^\lambda(t,\cdot)\to 0$ in $\L^1(0,+\infty)$,
    \item or there exists a unique $\lambda^*>0$ such that
    \[
        \lim_{t\to \infty} u^\lambda(t,\cdot) = \begin{cases}
            0, & \text{if}~ \lambda < \lambda^*, \\
            \overline{\varphi}_2, & \text{if}~ \lambda>\lambda^*,
        \end{cases}
    \]
    in $\L^1(0,+\infty)$. In that case, there exists $\delta>0$ with
    \[
        \liminf_{t\to \infty} \| u^{\lambda^*}(t,\cdot) \|_{\L^1} \geq \delta \quad \text{and} \quad \liminf_{t\to \infty} \| u^{\lambda^*}(t,\cdot) - \overline{\varphi}_2 \|_{\L^1} \geq \delta.
    \]
\end{itemize}
\end{theorem}

As far as we know, this is the first time that a sharp threshold dynamics for an age-structured model is rigorously proved. The proof of this result exploits the fact that having a monotone nonlinearity yields a monotone dynamical system \cite{MR1319817} in the state space. Although there is a considerably developed theory around this type of dynamical systems, the monotonicity on its own is not enough to determine the complete dynamics of the system. It must be strengthened as requiring a strong monotonicity or a strongly order-preserving property on the semiflow, the latter being a more flexible notion since it allows the cone generating the order to have an empty interior.

For the system \eqref{GM-model}, the main difficulty arises when we allow $\beta$ to have noncompact support; that is, $a^* = +\infty$. This means that the effect of the initial data $u_0$ over the solution is carried for all $t\geq 0$, and therefore the strong order preserving property for the semiflow generated by the model is not completely clear. Despite this, a particular non-compactly supported case can be studied by this approach. Namely, assume that $\beta$ and $\mu$ are eventually constant; that is, there are $a_0>0$ and positive constants $\beta_\infty$ and $\mu_\infty$ such that $\beta(a) = \beta_\infty$ and $\mu(a) = \mu_\infty$ for almost every $a>a_0$. Then, the dynamics of the Gurtin-MacCamy model can be studied employing a coupled integro-differential system of equations, which yields a strongly monotone dynamical system; see Section \ref{sec:non_compact}. Next, we state this claim formally.

\begin{assumption}\label{as:4}
There exists $a_0\geq 0$, $\beta_\infty>0$ and $\mu_\infty > 0$ such that $\beta(a) = \beta_\infty$ and $\mu(a) = \mu_\infty$ for almost every $a> a_0$.
\end{assumption}

\begin{theorem}\label{th:main_2}
Suppose Assumptions \ref{as-1}, \ref{as-2} and \ref{as:4} are valid. We have the following alternative:
\begin{itemize}
    \item Either for every $\lambda \geq 0$, $u(t, \cdot)\to 0$ in $\L^1(0,+\infty)$.
    \item or there exists a unique $\lambda^*>0$ with
    \[
        \lim_{t\to \infty} u^\lambda(t,\cdot) = \begin{cases}
            0, & \text{if}~ \lambda < \lambda^*, \\
            \overline{\varphi}_2, &\text{if}~ \lambda > \lambda^*,
        \end{cases}
    \]
    in $\L^1(0,+\infty)$. Additionally, there exists $\delta>0$ with
    \[
        \liminf_{t\to \infty} \| u^{\lambda^*}(t,\cdot) \|_{\L^1} \geq \delta \quad \text{and} \quad \liminf_{t\to \infty} \| u^{\lambda^*}(t,\cdot) - \overline{\varphi}_2 \|_{\L^1} \geq \delta.
    \]
\end{itemize}
\end{theorem}

The paper is organized as follows. In Section \ref{sec:prelim} we address the well-posedness question about the age-structured model, using two different approaches. First, we prove the existence of a continuous semiflow on $\L^1_+$ yielding the solution of the Gurtin-MacCamy model using integrated semigroups to solve the associated abstract formulation of the model. The local stability and instability of the steady states is also proved by means of the principle of linearised stability. The second approach uses the method of the characteristics to obtain a representation of the solution in terms of the solution of a nonlinear Volterra integral equation. We also state some properties of this solution, as its boundedness and a comparison principle deduced from the monotonicity assumption made on $f$. This formulation is the basis of the next development in the work. In Section \ref{sec:threshold} we handle the case when $\beta$ has compact support, letting us set the problem in terms of a semiflow defined on a closed subset of the space of continuous functions on some compact interval. In this section we describe the set, called the separatrix, which contains all the initial data producing an oscillatory solution around the intermediate equilibria. The separatrix turns out to be a totally unordered set, which implies the sharp threshold property and therefore our main result follows, as is proved in Section \ref{sec:main_th}. In Section \ref{sec:non_compact} we present how this technique can also be used to study a particular case when $\beta$ has no compact support, namely when $\beta$ and $\mu$ are eventually constant. Finally, the Appendix presents how the strongly order-preserving property fails when the birth rate is not compactly supported.

\section{Preliminaries}\label{sec:prelim}

\subsection{Abstract setting of the problem}

This subsection is intended to show an abstract framework for the problem in which the well-posedness can be established in a standard way. The theory of integrated semigroups has been extensively used to study semilinear abstract Cauchy problems where the operator is non-densely defined; we refer to \cite{MR3887640} and the references therein for a well and self-contained exposition of this topic and for further reading.

We consider the space of the initial distribution $u_0$ for the Gurtin-MacCamy model to be the Banach space of Lebesgue integrable functions $\L^1(0,+\infty)$, since its norm can be interpreted as the number of total members in the population at the initial time. Thus, set
\begin{align*}
    X &= \R \times \L^1(0,+\infty), \\
    X_0 & = \{0\} \times \L^1(0,+\infty), \\
    D(A) &= \{ 0 \} \times \W^{1,1}(0,+\infty),
\end{align*}
where these spaces are endowed with the usual product norm, and observe that $\overline{D(A)} = X_0$. Also, set $X^+ = \R_+ \times \L^1_+(0,+\infty)$ as the usual cone of non-negative elements of $X$ and similarly $X_0^+ = \{0\} \times \L^1_+(0,+\infty)$. Consider the linear operator $A\colon D(A)\subseteq X \to X$ defined by
\begin{equation}\label{def-A}
    A
    \begin{pmatrix}
        0 \\ \ph 
    \end{pmatrix} = 
    \begin{pmatrix}
        -\ph(0) \\ -\ph' - \mu\ph
    \end{pmatrix}.
\end{equation}
Note that $\ph(0)$ is well defined here owing to the Sobolev embedding $\W^{1,1}(0,+\infty) \hookrightarrow C(\R_+)$. Also, note that the operator is not densely defined since $\overline{D(A)} = X_0$.

Define $F\colon X_0 \to X$ by
\begin{equation}\label{def-F}
    F
    \begin{pmatrix}
        0 \\ \varphi
    \end{pmatrix} = 
    \begin{pmatrix}
        f\left( \int_0^{+\infty} \beta(a)\ph(a) \dd a \right) \\ 0
    \end{pmatrix}.
\end{equation}
Then, identifying the solution $u(t,a)$ of the age-structured model with $v(t) = \begin{pmatrix}
    0 \\ u(t,\cdot)
\end{pmatrix} \in D(A)$, it is possible to rewrite the problem \eqref{GM-model} as an abstract semilinear Cauchy problem
\begin{equation}\label{abs-prob}
    \df{v(t)}{t} = Av(t) + F(v(t)), \quad t>0, \quad v(0) = v_0 \in X_0^+.
\end{equation}

To deduce the existence and uniqueness of a maximal solution to this problem, we will apply the theory of integrated semigroups. In fact, these properties have already been mentioned in \cite[Section 8.3]{MR3987042}, but for the sake of completeness they are stated and proved here.

\begin{lemma}
The operator $A$ is a Hille-Yosida operator; that is, there exist $\omega\in \R$ such that $(\omega,+\infty) \subset \rho(A)$ and
\[
    \| (\lambda I - A)^{-n} \| \leq \frac{1}{(\lambda - \omega)^n}, \quad \forall \lambda > \omega, ~ \forall n\geq 1.
\]
In fact, it is possible to take $\omega = \esup(-\mu)<0$.
\end{lemma}
\begin{proof}
Set $\omega := \esup(-\mu)$. For $\lambda > \omega$ and an element $\vcol{\alpha \\ \psi}\in X$ the equation
\[
    (\lambda I - A)\begin{pmatrix}
        0 \\ \ph
    \end{pmatrix} = \begin{pmatrix}
        \alpha \\ \psi
    \end{pmatrix} \Leftrightarrow \begin{dcases}
        \ph(0) = \alpha, \\
        \lambda \ph + \ph' + \mu \ph = \psi, 
    \end{dcases}
\]
can be solved in $D(A)$ in a unique way as
\[
    \ph(a) = \alpha e^{-\int_0^a \lambda + \mu(l) \dd l} + \int_0^a e^{-\int_s^a \lambda + \mu(l) \dd l} \psi(s) \dd s.
\]
Moreover, we can estimate its norm as
\begin{align*}
    \int_0^{+\infty} |\ph(a)| \dd a & \leq |\alpha| \int_0^{+\infty} e^{-\int_0^a \lambda + \mu(l) \dd l} \dd a + \int_0^{+\infty} \int_0^a e^{-\int_s^a \lambda+\mu(l) \dd l} |\psi(s)| \dd s \dd a \\
    & \leq |\alpha| \int_0^{+\infty} e^{-(\lambda-\omega)a} \dd a + \int_0^{+\infty} \int_s^{+\infty} e^{-(\lambda-\omega)(a-s)} |\psi(s)| \dd s \dd a \\
    &= \frac{|\alpha|}{\lambda-\omega} + \frac{\|\psi\|}{\lambda - \omega} = \frac{1}{\lambda - \omega} \left\| \begin{pmatrix}
        \alpha \\ \psi
    \end{pmatrix} \right\|,
\end{align*}
and since $\| (\lambda I - A)^{-1} \vcol{\alpha \\ \psi} \| = \|\ph\|$ we deduce that
\[
    \| (\lambda I - A)^{-1} \| \leq \frac{1}{\lambda - \omega}
\]
which implies that $(\omega,+\infty)\subseteq \rho(A)$ and $A$ is a Hille-Yosida operator.
\end{proof}

Consider $A_0$, the part of $A$ if $X_0 = \overline{D(A)}$. This result imply that $A_0$ is a densely defined Hille-Yosida operator, and therefore it generates a $C_0$-semigroup $\{T_{A_0}(t)\}_{t\geq 0}$ on $X_0$ (see \cite[Lemma 3.4.2]{MR3887640}) which satisfies
\[
    \| T_{A_0}(t) \| \leq e^{\omega t},
\]
with $\omega$ as in the above Lemma. Furthermore, the semigroup generated by $A_0$ is given by the formula
\[
    T_{A_0}(t) \begin{pmatrix}
        0 \\ \ph
    \end{pmatrix} = \begin{pmatrix}
        0 \\ \hat{T}_{A_0}(t)\ph
    \end{pmatrix}, \quad \hat{T}_{A_0}(t)(\ph)(a) = \begin{cases}
           e^{-\int_{a-t}^a \mu(l) \dd l} \ph(a-t), & \text{for}~a\geq t, \\ 
           0, & \text{for}~ a\leq t,
        \end{cases}
\]
with $\begin{pmatrix} 0 \\ \varphi \end{pmatrix} \in X_0$. Indeed, note that if $\begin{pmatrix} 0 \\ \psi \end{pmatrix} \in X_0$, $\lambda > \omega$ and $a>0$ are given, then
\begin{align*}
    \int_0^{+\infty} e^{-\lambda s} \hat{T}_{A_0}(s)(\psi)(a) \dd s &= \int_0^a e^{-\lambda s} e^{-\int_{a-s}^a \mu (l) \dd l} \psi(a-s) \dd s \\
    &= \int_0^a e^{-\int_s^a \lambda + \mu(l) \dd l} \psi(s) \dd s.
\end{align*}
This computation implies that
\[
    (\lambda I - A_0)^{-1} \begin{pmatrix}
        0 \\ \psi
    \end{pmatrix} = \int_0^{+\infty} e^{-\lambda s} T_{A_0}(s) \begin{pmatrix}
        0 \\ \psi
    \end{pmatrix} \dd s, 
\]
and therefore the claim follows from the Laplace transform characterization of the infinitesimal generator due to Arendt (see \cite{MR872810} or \cite[Theorem 2.3.9]{MR3887640}).

Additionally, $A$ generates an integrated semigroup $\{S_A(t)\}_{t\geq 0}$ (see \cite[Proposition 3.4.3]{MR3887640}) which is given by
\[
    S_A(t) = (\lambda I - A_0) \int_0^t T_{A_0}(s) \dd s (\lambda I - A)^{-1},
\]
for any $\lambda>\omega$. Actually, when the integrated semigroup is restricted to the space $X_0$ it can be computed easily from the following representation
\[
    S_A(t)x = \int_0^t T_{A_0}(s)x \dd s, \quad \forall t\geq 0,~\forall x\in X_0.
\]

Thus, the integrated semigroup is defined by
\[
    S_A(t)\begin{pmatrix}
        \alpha \\  \ph
    \end{pmatrix} = \begin{pmatrix}
        0 \\ L(t)\alpha + \int_0^t \hat{T}_{A_0}(s)\ph \dd s
    \end{pmatrix},
\]
where
\[
    L(t)(\alpha)(a) = \begin{cases}
        0, & \text{for}~ a\geq t, \\
        \alpha e^{-\int_0^a \mu(l) \dd l}, & \text{for}~ a\leq t.
    \end{cases}
\]

On the other hand, since the function $f$ is assumed to be continuously differentiable with bounded derivative, it follows that $F$ is a positive Fréchet differentiable and Lipschitz map. Moreover, the derivative at an element $\vcol{0 \\ \ph}\in X_0$ is the linear operator $DF \vcol{0 \\ \ph} \colon X_0 \to X$ given by
\[
    DF\begin{pmatrix}
        0 \\ \ph
    \end{pmatrix} \begin{pmatrix}
        0 \\ \psi
    \end{pmatrix} = \begin{pmatrix}
        f'\left( \int_0^{+\infty} \beta(a) \ph(a) \dd a \right) \int_0^{+\infty} \beta(a) \psi(a) \dd a \\ 0
    \end{pmatrix}.
\]
Owing to these observations, and since the operator $A$ is resolvent positive, the following result follows from the theory of integrated semigroups applied to semilinear Cauchy problems, see e.g. \cite[Section 5]{MR3887640}.

\begin{proposition}
Let Assumption \ref{as-1} be satisfied. Then there exists a continuous semiflow $\{ U(t) \}_{t\geq 0}$ on $X_0^+$ such that for every $u_0\in X_0^+$, the map $t\mapsto U(t)u_0$ is the unique integrated solution of the abstract semilinear Cauchy problem \eqref{abs-prob}:
\[
    \begin{dcases}
    \df{v(t)}{t} = Av(t) + F(v(t)), \quad t\geq 0, \\
    v(0)= u_0,
    \end{dcases}
\]
that is,
\[
    U(t)u_0 = u_0 + A\int_0^t U(s)u_0 \dd s + \int_0^t F(U(s)u_0) \dd s.
\]
\end{proposition}

\subsection{Existence and local stability of steady states}

We start by reminding the principle of linearised stability for the evolution equation \eqref{abs-prob} with a non-densely defined generator; see \cite{MR872517,MR1073056}. Here, if $B$ represents the generator of a $C^0$-semigroup $\{ T_{B}(t) \}_{t\geq 0}$, define
\[
    \omega_0(B) := \lim_{t\to +\infty} \frac{\log \| T_B(t) \|}{t} \quad\text{and}\quad \omega_1(B) := \lim_{t\to +\infty} \frac{\log \| T_B(t) \|_{\text{ess}}}{t}
\]
as the growth bound and the essential growth bound of $\{T_B(t)\}_{t\geq 0}$, respectively, where $\| T_B(t) \|_{\text{ess}}$ denotes the essential seminorm of the operator $T_B(t)$.

\begin{proposition}
Let $F$ be continuously Fréchet differentiable in $X_0$ and let $u^*$ be a steady state. Denote by $(A+DF(v^*))_0$ the part of $A+DF(v^*)$ in $X_0$.
If $\omega_0((A+DF(v^*))_0)<0$ then for any $\omega>\omega_0((A+DF(v^*))_0)$ there exist $M,\delta> 0$ such that
\[
    \| U(t)u - u^* \| \leq M e^{\omega t} \| u-u^* \|, \quad \forall t\geq 0,
\]
for all $u\in X_0$ with $\|u-u^*\|\leq \delta$.
\end{proposition}

\begin{corollary}
Let $F$ be continuously Fréchet differentiable in $X_0$ and let $u^*$ be a steady state. Suppose that $\omega_1((A+DF(v^*))_0)<0$. If all the eigenvalues of $A+DF(v^*)$ have strictly negative real part then there exist $\omega<0$, $M,\delta>0$ such that
\[
    \| U(t)u - u^* \| \leq Me^{\omega t}\|u - u^*\|,\quad \forall t\geq 0,
\]
for all $u\in X_0$ with $\|u-u^*\|\leq \delta$. On the other hand, if at least one eigenvalue of $A+DF(u^*)$ has strictly positive real part, then $u^*$ is an unstable steady state.
\end{corollary}

It is easy to verify that the abstract problem \eqref{abs-prob} has exactly three equilibria
\[
    \overline{u}_0 = \begin{pmatrix}
        0 \\ 0
    \end{pmatrix}, \quad \overline{u}_1 = \begin{pmatrix}
        0 \\ \overline{\ph}_1
    \end{pmatrix}, \quad \text{and}\quad \overline{u}_2 = \begin{pmatrix}
        0 \\ \overline{\ph}_2
    \end{pmatrix},
\]
where $\overline{\ph}_i(a) = \kappa_i e^{-\int_0^a \mu(l) \dd l}$, $i\in\{1,2\}$. Regarding the local asymptotic stability of these steady states, we have the next result.

\begin{proposition}
The steady states $\overline{u}_0$ and $\overline{u}_2$ are locally asymptotically stable, while the intermediate equilibrium $\overline{u}_1$ is unstable.
\end{proposition}
\begin{proof}
We start with the stability analysis for the zero equilibrium $\overline{u}_0$. Since the semigroup generated by $A_0$ satisfies
\[
    \| T_{A_0}(t) \| \leq e^{\omega t},
\]
where $\omega = \esup(-\mu)\leq -\underline{\mu} <0$ with $\mu$ as in Assumption \ref{as-1}-(ii),  then it follows that $\omega_1(A_0)\leq \omega_0(A_0) \leq \omega$. As $DF(\overline{u}_0)$ is a bounded compact linear operator, from Theorem 1.2 in \cite{MR2394101} we get that $\omega_1( (A+DF(\overline{u}_0))_0 )\leq -\underline{\mu}$. Therefore, it remains to study the point spectrum of $(A+DF(\overline{u}_0))_0$ in the half-plane $\{ \lambda \colon \Re(\lambda)> -\underline{\mu} \}$.

The linearised equation of \eqref{abs-prob} around the trivial equilibrium $\overline{u}_0$ is
\[
    \df{v(t)}{t} = Av(t) + DF(\overline{u}_0)v(t),
\]
which corresponds to the following PDE formulation
\begin{equation}\label{eq:linear-model}
    \begin{dcases}
    \partial_t u(t,a) + \partial_a u(t,a) = -\mu(a) u(t,a), & t>0,~a>0, \\
    u(t,0) = f'(0) \int_0^{+\infty} \beta(a) u(t,a) \dd a, & t>0, \\
    u(0,\cdot) = u_0 \in \L^1(0,+\infty).
    \end{dcases}
\end{equation}

Next, we seek exponential solutions $u(t,a) = e^{\lambda t}u(a)$ to derive the characteristic equation. This yields the following system
\[
    \begin{dcases}
        u'(a) = -(\lambda + \mu(a))u(a), & a>0, \\
        u(0) = f'(0) \int_0^{+\infty}\beta(a) u(a) \dd a,
    \end{dcases}
\]
and after solving the differential equation and replacing this solution in the boundary condition, we obtain the characteristic equation
\begin{equation}\label{char-eq}
    f'(0) \int_0^{+\infty} \beta(a) e^{-\lambda a} e^{-\int_0^a \mu(l) \dd l} \dd a = 1.
\end{equation}

If $f'(0)=0$ the set of characteristic values is vacuous, then we assume $f'(0)>0$. Let us suppose that $\lambda$ is a root of this equation with a positive real part, since $f'(0)<1$ it follows that
\[
    1 < \frac{1}{|f'(0)|} \leq \int_0^{+\infty} \beta(a) e^{-\Re(\lambda) a} e^{-\int_0^a \mu(l) \dd l} \dd a < \int_0^{+\infty} \beta(a) e^{-\int_0^a \mu(l) \dd l} \dd a = 1,  
\]
which is absurd. On the other hand, if $\lambda = i\xi$, $\xi\in \R$, is a solution to \eqref{char-eq} then by splitting in the real and imaginary parts we get
\[
    f'(0) \int_0^{+\infty} \beta(a) e^{-\int_0^a \mu(l) \dd l} \cos(\xi a) \dd a = 1,
\]
which leads us to a contradiction just as before.

The linearised equations of \eqref{abs-prob} around $\overline{u}_1$ and $\overline{u}_2$ are obtained in the same way, and they correspond to the same PDE formulation as in \eqref{eq:linear-model} changing $f'(0)$ by $f'(\kappa_i)$, with $i\in\{1,2\}$ as corresponds. Therefore, the characteristic equation for $\overline{u}_2$ is
\[
    f'(\kappa_2) \int_0^{+\infty} \beta(a) e^{-\lambda a} e^{-\int_0^a \mu(l) \dd l} \dd a = 1,
\]
and, due to $f'(\kappa_2)<1$, the stability of $\overline{u}_2$ follows by the same arguments as for the zero equilibrium.

On the other hand, the characteristic equation for $\overline{u}_1$ is
\[
    f'(\kappa_1) \int_0^{+\infty} \beta(a) e^{-\lambda a} e^{-\int_0^a \mu(l) \dd l} \dd a = 1.
\]
Note that the left-hand side of this equation defines a continuous and decreasing function $G \colon (-\underline{\mu},+\infty) \to \R$ such that $G(0)=f'(\kappa_1)>1$ and $G(+\infty) = 0$. Therefore, the characteristic equation has a unique positive real root, which implies the instability of $\overline{u}_1$.
\end{proof}

\subsection{Volterra formulation}
Besides the formulation presented in the previous sections, it is also possible to give an almost explicit representation of the solution to the Gurtin-MacCamy system. In fact, by solving the first order partial differential equation of \eqref{GM-model} by the method of the characteristics, we obtain the following formula for the solution:
\begin{equation}\label{charac-1}
    u(t,a) = \begin{dcases}
        e^{-\int_{a-t}^a \mu(l) \dd l} u_0(a-t), & \text{for}~a\geq t, \\
        e^{-\int_0^a \mu(l) \dd l} b(t-a), & \text{for}~ a\leq t,
    \end{dcases}
\end{equation}
where $b\colon \R_+ \to \R$ is the unique continuous solution to the nonlinear Volterra equation
\begin{equation}\label{eq:int_1}
    b(t) = f\left( \int_t^{+\infty} \beta(a) e^{-\int_{a-t}^a \mu(l) \dd l} u_0(a-t) \dd a + \int_0^t \beta(a)e^{-\int_0^a \mu(l) \dd l} b(t-a) \dd a \right).
\end{equation}

This representation is equivalent to
\begin{equation}\label{charac-2}
    u(t,a) = \begin{dcases}
        e^{-\int_{a-t}^a \mu(l) \dd l} u_0(a-t), & \text{for}~a\geq t, \\
        e^{-\int_0^a \mu(l) \dd l} f(B(t-a)), & \text{for}~ a\leq t,
    \end{dcases}
\end{equation}
where $B\colon \R_+ \to \R$ is the unique continuous solution to the nonlinear Volterra integral equation
\begin{equation}\label{eq:int_2}
    B(t) = \int_t^{+\infty} \beta(a) e^{-\int_{a-t}^a \mu(l) \dd l} u_0(a-t) \dd a + \int_0^t \beta(a)e^{-\int_0^a \mu(l) \dd l} f(B(t-a)) \dd a,
\end{equation}
and the relation is given by 
\begin{align*}
    b(t) &= f(B(t)), \\
    B(t) &= \int_t^{+\infty} \beta(a) e^{-\int_{a-t}^a \mu(l) \dd l} u_0(a-t) \dd a + \int_0^t \beta(a)e^{-\int_0^a \mu(l) \dd l} b(t-a) \dd a.
\end{align*}

Thus, in order to understand the long-term dynamics of the problem, we need to understand it for any of the two versions of the integral equations just shown. Equations of the form \eqref{eq:int_1}, as well as their integro-differential forms, appear recurrently in different scenarios of population dynamics and epidemiology, e.g. \cite{MR682251,MR312923,MR666008,MR1050319} for some examples. Concerned about the long-term behavior of the solutions to this equation, Londen and Levin, independently, in a series of works (see \cite{MR304994,MR458084} and the references therein), proved that under a monotonicity condition on the kernel and feeble conditions on the remaining parameters, all bounded solutions tend to converge towards the set of equilibrium points of the function $f$. Recently, Herrera and Trofimchuk \cite{MR4750985} studied this integral equation for the monostable case, obtaining conditions for absolute convergence of the solutions towards the nontrivial equilibrium (here, by absolute convergence we mean they are independent of the choice of the kernel). In this work we are concerned with the bistable case, but some of those results are still applicable.

The following proposition contains some basic properties of the solutions of \eqref{eq:int_1}.

\begin{proposition}\label{prop:basics}
Let Assumption \ref{as-1} be satisfied. Then there exists a unique continuous solution $b\colon \R_+ \to \R$ of equation \eqref{eq:int_1}, and this solution can be obtained by the method of successive approximations; that is, given any $b_0\in C(\R_+)$, the solution $b$ results from the limit of the iterative schema $b_{n+1} = \AA(b_n)$, where $\AA\colon C(\R_+) \to C(\R_+)$ is the operator defined by the right-hand side of \eqref{eq:int_1}; namely,
\[
    \AA(b)(t) = f\left( \int_t^{+\infty} \beta(a) e^{-\int_{a-t}^a \mu(l) \dd l} u_0(a-t) \dd a + \int_0^t \beta(a) e^{-\int_0^a \mu(l) \dd l} b(t-a) \dd a \right).
\]
Moreover, the solution $b$ is uniformly continuous on $\R_+$. 

Also, we have continuous dependence on the initial data in the following sense: given $\varepsilon>0$, $u_0\in \L^1_+(0,+\infty)$ and $T>0$, then there exists $\delta>0$ such that for every $v_0\in \L^1_+(0,+\infty)$ with $\|u_0-v_0\| < \delta$, for the associated solutions $b^u, b^v$ of \eqref{eq:int_1} it holds that $|b^v(t) - b^u(t)| < \varepsilon$ for every $t\in [0,T]$.
\end{proposition}
\begin{proof}
The existence of the solution follows from a straightforward application of the contraction mapping principle considering the operator $\AA$ in an appropriate weighted space of continuous functions; see for instance \cite{MR4750985,Ruiz} or Chapter 2 in \cite{MR4461043}. The uniform continuity follows from the regularizing properties of the convolution operation and the continuous dependence by a straightforward application of the Grönwall's inequality.
\end{proof}

Moreover, the subdiagonal condition given in Assumption \ref{as-1} yields the uniform boundedness of the solutions (c.f. \cite{MR361694}).

\begin{theorem}\label{th:bound}
Let Assumption \ref{as-1} be satisfied. The solutions of the equation are uniformly ultimately bounded; that is, there exists $M>0$ such that for every function $u_0 \in \L^1_+(0,+\infty)$, the associated solution $b = b(\cdot; u_0)$ of \eqref{eq:int_1} satisfies
\[
    \limsup_{t\to \infty} b(t) \leq M.
\]
More precisely, let $b^* = \limsup_{t\to \infty} b(t)$, and $b_* = \liminf_{t\to \infty} b(t)$, then
\[
    b^* \in \{0\}\cup [\kappa_1,\kappa_2] \quad \text{and} \quad b_* \in [0,\kappa_1] \cup \{\kappa_2\}.
\]
Additionally, there are constants $C_1,C_2>0$ depending only on $f$ and $\beta$ such that
\[
    B(t) \leq C_1 + C_2 \|u_0\|, \quad t\geq 0.
\]
\end{theorem}
\begin{proof}
We will prove the result for the equation \eqref{eq:int_2} for simplicity. First, let $\rho\in (0,1)$ be such that
\[
    \limsup_{x\to \infty} \frac{f(x)}{x} \leq \rho,
\]
then there exists $M>0$ such that $f(x)\leq \rho x$ for every $x\geq M$. Now, let us start by showing that every solution of \eqref{eq:int_2} is bounded. To accomplish this, let $T>0$ be any positive number and $B$ be any solution of \eqref{eq:int_2}. Define the set $K= \{ t\geq 0 \colon B(t) > M\}$, and note that for $t\in [0,T]$ there holds
\begin{align*}
    B(t) & \leq \ovbeta e^{-\unmu t} \|u_0\|_{\L^1} + \left( \int_{[0,t]\cap K^c} + \int_{[0,t] \cap K} \right) \beta(t-a) e^{-\int_0^{t-a} \mu(l) \dd l} f(B(a)) \dd a \\
    & \leq \ovbeta \|u_0\|_{\L^1} + \left[ f(M) + \rho \max_{t\in[0,T]} B(t) \right] \int_0^t \beta(t-a) e^{-\int_0^{t-a} \mu(l) \dd l} \dd a.
\end{align*}
As the last integral is bounded from above by 1 thanks to the normalization hypothesis \eqref{eq:norm}, and this holds for every $t\in [0,T]$, we infer
\[
    \max_{t\in [0,T]} B(t) \leq \ovbeta \|u_0\|_{\L^1} + f(M) + \rho \max_{t\in [0,T]} B(t),
\]
or equivalently
\[
    \max_{t\in[0,T]} B(t) \leq \frac{\ovbeta \|u_0\|_{\L^1} + f(M)}{1-\rho}.
\]
Since the right-hand side in the last expression is independent of $T$, the boundedness of the solution follows.

Now, let $B^* = \limsup_{t\to \infty} B(t)$, $\varepsilon>0$ and consider $T_\varepsilon>0$ such that $B(t) \leq B^* + \varepsilon$ for every $t\geq T_\varepsilon$. Additionally, take $(t_n)_{n\in \N}$ a sequence of real numbers such that $t_n\to \infty$ and $B(t_n) \to B^*$. Then, for any $n$ large enough we have
\begin{align*}
    B(t_n) &\leq \ovbeta e^{-\unmu t_n} \|u_0\|_{\L^1} + \left( \int_0^{t_n-T_\varepsilon} + \int_{t_n-T_\varepsilon}^{t_n} \right) \beta(a) e^{-\int_0^a \mu(l) \dd l} f(B(t_n-a)) \dd a \\
    & \leq \ovbeta e^{-\unmu t_n} \| u_0 \|_{\L^1} + f(B^* + \varepsilon) \int_0^{t_n-T_\varepsilon} \beta(a) e^{-\int_0^a \mu(l) \dd l} \dd a + K \int_{t_n - T_\varepsilon}^{t_n} \beta(a) e^{-\int_0^a \mu(l) \dd l} \dd a \\
    &\leq \ovbeta e^{-\unmu t_n} \| u_0 \|_{\L^1} + f(B^* + \varepsilon) \int_0^{t_n-T_\varepsilon} \beta(a) e^{-\int_0^a \mu(l) \dd l} \dd a + K\ovbeta \int_{t_n - T_{\varepsilon}}^{t_n} e^{-\unmu a} \dd a, 
\end{align*}
where $K$ is some bound for $f(B(t))$. Then, taking $n\to \infty$ we deduce that $B^*\leq f(B^* + \varepsilon)$ and therefore, since $\varepsilon$ is arbitrary, it follows that $B^* \leq f(B^*)$. Because of the geometry of the function $f$ and the last inequality, we infer that $B^* \in \{0\} \cup [\kappa_1, \kappa_2]$ as stated. 

For the result concerned about the limit inferior, let $B_* = \liminf_{t\to\infty} B(t)$, $\varepsilon>0$ be given and $T_\varepsilon>0$ such that $B(t) \geq B_* - \varepsilon$ for every $t\geq T_\varepsilon$. Taking a sequence of times $(t_n)_{n\in \N}$ such that $B(t_n)$ tends to $B_*$, then just as before, we get
\[
    B(t_n) =\int_{t_n}^{+\infty} \beta(a) e^{-\int_{a-t}^a \mu(l) \dd l} u_0(a-t) \dd a + \left( \int_0^{t_n-T_\varepsilon} + \int_{t_n - T_\varepsilon}^{t_n} \right) \beta(a) e^{-\int_0^a \mu(l) \dd l} f(B(t_n - a)) \dd a.
\]
The first and third terms in the last expression go to 0 as $n\to \infty$ in the same fashion as before, whereas the second term is bounded below by $f(B_* - \varepsilon) \int_0^{t_n -T_\varepsilon} \beta(a) e^{-\int_0^a \mu(l) \dd l} \dd a$. Consequently, after passing to the limit, we obtain that $B_* \geq f(B_*)$, and the conclusion follows because of the geometry of the function $f$.
\end{proof}

\subsection{Comparison principle}

There exist some results concerning the comparison principle in the Gurtin-MacCamy model; see, e.g., \cite{MR3987042,Ruiz}. For the aim of our purposes, we establish a simpler version of the comparison principle whose proof follows immediately from the successive approximation method to construct the solutions of the integral equation. Mainly we have:

\begin{proposition}(Comparison principle)
Let Assumption \ref{as-1} be satisfied. If $u_0,v_0\in \L^1_+(0,+\infty)$ are such that $u_0 \leq v_0$ a.e, then the associated solutions to the integral equation $b^u = b(\cdot ; u_0)$ and $b^v = b(\cdot ; v_0)$ verify
\[
    b_u(t) \leq b_v(t), \quad \forall t\geq 0.
\]
\end{proposition}
\begin{proof}
Let $u_0, v_0\in \L^1_+(0,+\infty)$ be given and such that $u_0 \leq v_0$ a.e. Denote $b^u$ and $b^v$ to the solutions of \eqref{eq:int_1} corresponding to $u_0$ and $v_0$, respectively. Thus, from the monotonicity assumption over $f$, it follows that for each $t\geq 0$
\begin{align*}
    b^u(t) &= \AA_u(b^u)(t) \\
    &= f\left( \int_t^{+\infty} \beta(a) e^{-\int_{a-t}^a\mu(l) \dd l} u_0(a-t) + \int_0^t \beta(a) e^{-\int_0^a \mu(l) \dd l} b^u(t-a) \dd a \right) \\
    &\leq f\left( \int_t^{+\infty} \beta(a) e^{-\int_{a-t}^a\mu(l) \dd l} v_0(a-t) + \int_0^t \beta(a) e^{-\int_0^a \mu(l) \dd l} b^u(t-a) \dd a \right) = \AA_v(b^u)(t).
\end{align*}
Here, we have used the subscript in the operator $\AA$ to denote its dependence on the initial distributions $u_0$ and $v_0$. Using the last inequality we get by induction that $b^u \leq \AA_v^n(b^u)$ for every $n\in \N$, and hence the conclusion follows from Proposition \ref{prop:basics}.
\end{proof}

This principle can be used to establish the asymptotic behavior of some particular solutions to the integral equation. In particular, due to the monotonocity and the one-dimensional dynamics of the map $f$, one expects that the solutions starting from below or above the unstable equilibrium will converge to the trivial or to the stable nontrivial equilibrium, respectively. As a matter of fact, this observation can be proved from the comparison principle.

\begin{corollary}\label{cor:conv}
Let Assumption \ref{as-1} be satisfied. Assume in addition that $u_0\in \L^1_+(0,+\infty)$ is such that
\[
    u_0(a) \geq (\leq) \kappa_1 e^{-\int_0^a \mu(l) \dd l}, \quad \text{a.e. for}~a\geq 0,
\]
and such that $\supp u_0 \cap [0,a^*] \neq \emptyset$. Then, the corresponding solution of the system \eqref{GM-model} satisfies $u(t,\cdot) \to \overline{\varphi_2}$ ($0$) in $\L^1(0,+\infty)$. 
\end{corollary}
\begin{proof}
Let us denote by $b$ the corresponding solution to the equation \eqref{eq:int_1} associated with the given initial function $u_0$. Using the changes of variables $\xi(t) = b(t)- \kappa_1$, $g(x) = f(x+\kappa_1)-\kappa_1$, and $\eta_0(a) = u_0(a) - \kappa_1 e^{-\int_0^a \mu(l) \dd l}$ we observe that $\xi$ is a continuous solution of the following equation
\[
    \xi(t) = g\left( \int_t^{+\infty} \beta(a) e^{-\int_{a-t}^a \mu(l) \dd l} \eta_0(a-t) \dd a + \int_0^t \beta(a) e^{-\int_0^a \mu(l) \dd l} \xi(t-a) \dd a \right).
\]
Furthermore, $g$ has a monostable structure  and, due to the comparison principle, $\xi$ is a nonzero solution. Then, the claim of the Theorem follows from the uniform persistence result from \cite{MR4750985} (see Theorem 1 therein) and Theorem \ref{th:bound}. The other case is analogous. 
\end{proof}

\section{The threshold set}\label{sec:threshold}

Observe that the integral equation \eqref{eq:int_1} can be rewritten as
\[
    b(t) = f\left( \int_0^{+\infty} \beta(a) e^{-\int_0^a \mu(l) \dd l} b(t-a) \dd a \right),
\]
where we have defined
\[
    b(t) = e^{\int_0^{-t} \mu(l) \dd l} u_0(-t), \quad \text{for}~t<0.
\]
Hereinafter, given $t\in \R$ and $\varphi\colon \R \to \R$ we denote by $\varphi_t$ to the history of the function $\varphi$ defined by $\varphi_t(s) = \varphi(t+s)$, $s<0$; the domain of the parameters will be clear from the context.

Now, as $\beta$ vanishes for values greater than $a^*$ (see \eqref{def:a*}), the previous equation reduces to 
\begin{equation}\label{eq:int_3}
    b(t) = f\left( \int_0^{a^*} \beta(a)e^{-\int_0^a \mu(l) \dd l} b(t-a) \dd a  \right), \quad t\geq 0,
\end{equation}
and therefore the initial values of $b$ need to be specified on $(-a^*,0)$. Let us suppose that $a^*<+\infty$ and set $C_+^0 := C_+^0([-a^*,0])$ the set of non-negative continuous functions. Given a function $\varphi \in C_+^0$ and taking $b$ as the solution to the associated equation, we already know that b is a continuous function on $\R_+$, but note that 
\[
    b(0) = f\left( \int_0^{a^*} \beta(a) e^{-\int_0^a \mu(l) \dd l} \varphi(-a) \dd a \right) \neq \varphi(0),
\]
thus, the function $b^\varphi \colon (-a^*,+\infty) \to \R_+$ defined by 
\[
    b(t) = \begin{cases}
        b(t), & t\geq 0, \\
        \varphi(t), & t\in(-a^*,0),
    \end{cases}
\]
may have a discontinuity at zero, and therefore it is not possible to define a semiflow on the whole space $C([-a^*,0])$. Nevertheless, taking the following closed subset $X\subseteq C_+^0$ as
\begin{equation}\label{def:X}
    X = \left\{ \varphi\in C_+^0 \colon f\left( \int_0^{a^*} \beta(a) e^{-\int_0^a \mu(l) \dd l} \varphi(-a) \dd a \right) = \varphi(0) \right\},
\end{equation}
then every initial function $\varphi\in X$ yields a continuous solution of \eqref{eq:int_3} and therefore we can consider the continuous semiflow $\Phi \colon \R_+ \times X \to X$ defined by $\Phi(t,\varphi) = b^\varphi_t$, where $b^\varphi$ denotes the solution to the integral equation equipped with the initial data $\varphi$. This turns out to be a monotone semiflow, due to the comparison principle from the previous section, and it is clear that the equilibrium points for the semiflow are the constant functions given by the fixed points of $f$. The complete metric space $X$ is endowed with the order relation given by the cone of $C_+^0$. In the following, we will state some basic properties of the semiflow.

\begin{proposition}\label{compact_semifl}
Suppose that the conditions on Assumption \ref{as-1} hold. Then, the semiflow $\Phi$ is eventually strongly monotone and eventually compact. Namely, the following properties hold
\begin{enumerate}[label=(\roman*)]
    \item If $\varphi < \psi$ are two elements of $X$, then there is some $t_0>0$ such that $\Phi(t_0,\varphi) \ll \Phi(t_0,\psi)$. As a matter of fact, $\Phi(t, \varphi) \ll \Phi(t,\psi)$ for all $t\geq t_0$.
    \item For every bounded subset $B\subseteq X$, $\Phi(t,B)$ has compact closure for any $t\geq a^*$.
\end{enumerate}
\end{proposition}
\begin{proof}
For the first part, let $\varphi,\psi\in X$ such that $\varphi<\psi$, and set $\gamma(a) =  \beta(a)e^{-\int_0^a \mu(l) \dd l}$. As the solutions $b^\varphi$ and $b^\psi$ are bounded, it is possible to take $K>0$ large enough such that the integral term in $\eqref{eq:int_3}$ stays in $[0,K]$ for all $t\geq 0$, both for $\varphi$ and $\psi$. Then, thanks to Assumption \ref{as-1}-(iii), we can choose $\delta>0$ such that $f'(x)\geq \delta$ for all $x\in [0,K]$. Thus, setting $\xi(t) = b^\psi(t) - b^\varphi(t)$, it follows for $t\geq 0$ that
\begin{align*}
    \xi(t) &\geq \delta \int_0^{a^*} \gamma(a) \xi(t-a) \dd a \\
    &= \delta \sigma(t) + \delta \int_0^t \gamma(a) \xi(t-a) \dd a, 
\end{align*}
where
\[
    \sigma(t) = \begin{cases}
        \int_t^{a^*} \gamma(a) [\psi(t-a) - \varphi(t-a)] \dd a, & t\in[0,a^*], \\
        0, & t>a^*.
    \end{cases}
\]
Because of the hypothesis made about the order relation between $\psi$ and $\varphi$, we can see that $\sigma \not\equiv 0$. In fact, there is an interval $(\underline{\alpha},\overline{\alpha}) \subseteq (-a^*,0)$ such that $\psi(s) - \varphi(s) > 0$ for all $s\in(\underline{\alpha},\overline{\alpha})$. Thus, by taking $t=\underline{\alpha} + a^*$ and due to the definition of $a^*$ follows 
\[
    \sigma(t) = \int_{\underline{\alpha} + a^*}^{a^*} \gamma(a) [\psi(\underline{\alpha} + a^* -a) - \varphi(\underline{\alpha} + a^*) ] \dd a \neq 0.
\]
Therefore, from Corollary B.6 in \cite{MR2731633} we infer there is some $t_0>0$ such that $\xi(t)>0$ for all $t>t_0$, which implies that $b^\varphi_t \ll b^\psi_t$ for all $t> t_0 + a^*$.

For the next claim, it is enough to show that $\Phi(a^*,\cdot)$ maps bounded sets into relatively compact ones due to the semiflow property. Let us take $\{ \varphi_n \}_{n\in \N}$ a bounded sequence in $X$, and note by $b_n$ the associated solutions. Also, let $t\in [0,a^*)$ and $h>0$ such that $t+h\leq a^*$. Then, let $L$ be a Lipschitz constant for $f$ and observe
\begin{align*}
    |b_n(t+h) - b_n(t)| &\leq L \left| \int_{t+h-a^*}^{t+h}  \gamma(t+h-a) b_n(a) \dd a - \int_{t-a^*}^t \gamma(t-a) b_n(a) \dd a \right| \\
    & \leq L\Big[ \int_t^{t+h} \gamma(t+h-a) b_n(a) \dd a + \int_{t+h-a^*}^t |\gamma(t+h-a) - \gamma(a)| b_n(a) \dd a  \\
    & \qquad + \int_{t-a^*}^{t+h-a^*} \gamma(t-a) b_n(a) \dd a\Big].
\end{align*}
Now, Theorem \ref{th:bound} and the boundedness of the sequence $\{\varphi_n\}_{n\in\N}$ imply the solutions $b_n$ are uniformly bounded on $\R_+$. This fact, jointly with the continuity of the shift operator in $\L^1$, leads to the equicontinuity of the family $\{b_n\}_{n\in \N}$ as functions defined on $[0,a^*]$. Thus, the claim follows from the Arzelá-Ascoli Theorem.
\end{proof}

\begin{corollary}\label{attractor}
Under the same assumptions as in the previous proposition, the semiflow $\Phi\colon \R_+ \times X \to X$ is asymptotically smooth. Furthermore, there exists $A\subset X$ a compact attractor of bounded sets for the semiflow $\Phi$.
\end{corollary}
\begin{proof}
The first statement is a direct consequence of the last result, and jointly with the boundedness result given in the previous section, it implies the existence of the compact attractor of bounded sets (see Theorem 2.33 in \cite{MR2731633}).
\end{proof}

\subsection{The separatrix}

From Corollary \ref{cor:conv} we note that the unknown dynamics of \eqref{eq:int_3} come from those initial functions producing oscillating solutions around the unstable equilibrium $\kappa_2$. This situation has also been encountered in the study of delayed differential equations arising in neural network theory, for instance, during the modeling of single neurons as well as for synchronized activities in those networks \cite{MR1719128,MR1834537}.

This leads to considering the following set
\[
    S = \{ \varphi\in X \colon (b^\varphi)^{-1}(\kappa_1) ~ \text{is unbounded from above} \},
\]
whose elements are the initial data producing a solution that intersects with the intermediate equilibrium $\kappa_1$ for arbitrarily large times, cf. \cite{MR1822211,MR1719128}. Clearly, $S\neq \varnothing$ since it contains the constant function with the value of the equilibrium $\kappa_1$; also, $S$ is a closed set. Indeed, if $\varphi\in X\setminus S$, then we may assume without loss of generality that the associated solution $b^\varphi$ satisfies $b^\varphi_t \ll \kappa_1$ for some $t>0$ large enough. Therefore, as a consequence of the continuous dependence on the initial data (see Proposition \ref{prop:basics}) there exists an open neighborhood of $\varphi$ contained in $X\setminus S$. In particular, every solution to the integral equation associated with an initial function in this neighborhood goes to 0 as $t\to \infty$ because of Corollary \ref{cor:conv}.

The set $S$ acts as a separatrix for the semiflow, in the sense that it divides the state space $X$ in the initial functions whose associated solutions will eventually stay strictly above (below) the unstable equilibrium, and therefore converge to the stable nontrivial (trivial) equilibrium. Thus, if we define the sets
\begin{align*}
    X_0 = \{ \varphi\in X \colon b^\varphi_t \ll \kappa_1 ~\text{for some}~t>0 \}, \\
    X_{\kappa_2} = \{ \varphi\in X \colon b^\varphi_t \gg \kappa_1 ~\text{for some}~t>0 \},
\end{align*}
then we have the next decomposition of the state space as $X = X_0 \sqcup S \sqcup X_{\kappa_2}$.

\begin{proposition}\label{unord_S}
The set $S$ is totally unordered; i.e., if $\varphi,\psi\in S$ and $\varphi\leq \psi$, then $\varphi = \psi$.
\end{proposition}
\begin{proof}
We proceed by contradiction. Let $\varphi,\psi\in S$ with $\varphi < \psi$. The eventual strong monotonicity of the semiflow yields the existence of some $t_0>0$ such that $b^\varphi_{t_0} \ll b^\psi_{t_0}$. As the quasiconvergence is generic in strongly order-preserving semiflows (see Theorem 4.3 of Chapter 1 in \cite{MR1319817}; see also \cite{MR1112067}) we can find two elements $\tilde{\varphi},\tilde{\psi}\in X$ such that $b^\varphi_{t_0} \ll \tilde{\varphi} \ll \tilde{\psi} \ll b^\psi_{t_0}$, and their $\omega$-limit sets verify $\omega(\tilde{\varphi}),\omega(\tilde{\psi}) \subseteq \{ 0, \kappa_1, \kappa_2 \}$. Since the $\omega$-limit sets are connected, these elements produce a convergent solution of the equation. In particular, note that none of them can converge to $0$ or $\kappa_2$ by reason of the oscillatory behavior of $b^\varphi$ and $b^\psi$.

As a consequence, $b^{\tilde{\varphi}}$ and $b^{\tilde{\psi}}$ converge to $\kappa_1$. Let $\xi(t) = b^{\tilde{\varphi}}(t) - b^{\tilde{\psi}}(t)$ and, as before, $\gamma(a) = \beta(a) e^{-\int_0^{a} \mu(l) \dd l}$ and note that
\[
    \xi(t) = g(t) \int_0^{a^*} \gamma(a) \xi(t-a) \dd a,
\]
where
\[
    g(t) = \int_0^1 f'\left( s\int_0^{a^*} \gamma(a) b^{\tilde{\psi}}(t-a) \dd a + (1-s) \int_0^{a^*} \gamma(a) b^{\tilde{\varphi}}(t-a) \dd a \right) \dd s.
\]
The convergence of these solutions implies that $g(t)\to f'(\kappa_1)>1$ as $t\to \infty$, and in consequence, we can take $\rho\in(1,f'(\kappa_1))$ and $T>0$ large enough to satisfy $g(t)\geq\rho$ for all $t\geq T$. Then, we obtain
\[
    \xi(t) \geq \rho \int_0^{a^*} \gamma(a) \xi(t-a) \dd a, 
\]
but this is a contradiction. Indeed, as $\xi(t)>0$ for all $t\geq 0$ and it vanishes at infinity, we can choose a sequence $(t_n)_{n\in \N}$ such that $t_n\to \infty$ and $\xi(t_n) = \min_{s\in[0,t_n]} \xi(s)$. Thus, from the above inequality follows $\xi(t_n) \geq \rho \xi(t_n)$ for all $n$ large enough, which is impossible since $\rho>1$. The contradiction concludes the proof of the result.
\end{proof}

\subsection{The sharp threshold property}

In what follows, we are interested in the long-term behavior for a monotone family of initial functions $\{\varphi_\lambda\}_{\lambda \geq 0}$ in $X$. More precisely, we adopt the following assumption.

\begin{assumption}\label{as-3}
The family of initial functions $\{ \varphi_\lambda \}_{\lambda \geq 0}$ in $X$ is such that the map $\lambda \mapsto \varphi_\lambda$ is continuous and strictly monotone as a map from $\R_+$ into $\L^1(0,+\infty)$ i.e., if $\lambda < \eta$ then $\varphi_\lambda < \varphi_\eta$. Furthermore, $\varphi_0 \equiv 0$.
\end{assumption}

For each $\lambda\geq 0$, we denote by $b^\lambda$ the associated solution of \eqref{as-3} with the initial function $u_\lambda$. With such a family in $X$, we can split the index set employing the following subsets
\begin{align*}
    M_0 &= \{ \lambda\in \R_+ \colon b^\lambda(t) \to 0~\text{as}~ t\to\infty\}, \\
    M_{\kappa_2} &= \{ \lambda\in \R_+ \colon b^\lambda(t) \to \kappa_2, ~\text{as}~t\to \infty \}, \\
    M_\sigma &= \R_+ \setminus (M_0 \cup M_{\kappa_2}).
\end{align*}
The monotonicity, eventual strong monotonicity of the semiflow, and the continuous dependence on the initial data easily show that the sets $M_0$ and $M_{\kappa_2}$ are open intervals, the former one being of the form $[0,\lambda_*)$ with $\lambda_*\in (0,+\infty]$, and the last one being either empty or of the form $(\lambda^*, +\infty)$ with $\lambda^*\in(\lambda_*,+\infty)$. From the previous subsection, we know that if $\lambda\in M_\sigma$ then $\varphi_\lambda \in S$, and therefore, $\lambda_* = \lambda^*$ due to the nonordering of $S$. We refer to this equality as the sharp threshold property.

\begin{theorem}\label{th:threshold}
Let Assumption \ref{as-3} be satisfied. The set $M_\sigma$ is either empty or a single point; that is, either each of the solutions $b^\lambda$ vanishes at infinity, or there exists a unique $\lambda^*>0$ such that
\begin{itemize}
    \item for all $\lambda < \lambda^*$, $b^\lambda(t) \to 0$ as $t\to \infty$.
    \item for all $\lambda > \lambda^*$, $b^\lambda(t)\to \kappa_2$ as $t\to\infty$.
\end{itemize}
Furthermore, in the latter case, there exists $\delta>0$ with
\[
    \delta \leq \liminf_{t\to \infty} b^{\lambda^*}(t) \leq \limsup_{t\to \infty} b^{\lambda^*}(t) \leq \kappa_2 - \delta.
\]
\end{theorem}
\begin{proof}
The first part of the proof follows from the above discussion. For the last claim, note that $\varphi_{\lambda^*}\in S$, and since $S$ is a closed and forward invariant set, then $\omega(\varphi_{\lambda^*}) \subseteq A \cap S$, where $A$ is the compact attractor for $\Phi$. Consequently, noting that $A\cap S$ is compact and $\{0, \kappa_2\} \cap (A\cap S) = \varnothing$ we get the conclusion of the theorem.
\end{proof}

\section{Proof of the main result}\label{sec:main_th}

This section is dedicated to the proof of Theorem \ref{th:main}. Recall that the semiflow $U(t) \colon \L^1_+(0,+\infty) \to \L^1_+(0,+\infty)$ defined by the Gurtin-MacCamy model can be expressed in terms of the solution formula given by the method of the characteristics; i.e., $U(t)u_0 = u(t,\cdot)$, where
\[
    u(t,a) = \begin{cases}
        e^{-\int_{a-t}^a \mu(l) \dd l} u_0(a-t), & a\geq t, \\
        e^{-\int_0^a \mu(l) \dd l} b(t-a), & a\leq t,
    \end{cases}
\]
where $b\colon \R_+ \to \R_+$ is the unique continuous solution of the integral equation (see Section \ref{sec:threshold})
\begin{equation}\label{eq:int-proof}
    b(t) = f\left( \int_0^{a^*} \beta(a) e^{-\int_0^a \mu(l) \dd l} b(t-a) \dd a \right), \quad t\geq 0,
\end{equation}
with the specific initial values
\[
    b(t) = e^{\int_0^{-t} \mu(l) \dd l} u_0(-t), \quad t\in(-a^*, +\infty).
\]

Now, let us suppose that $a^*<+\infty$. Then, the solution $b$ is completely determined by the values of the initial data $u_0$ in the interval $(0,a^*)$. Therefore, the dynamics of the semiflow $\{ U(t) \}_{t\geq 0}$ on $\L^1$ relies on the dynamics of the semiflow $\{ \Phi(t,\cdot ) \}_{t\geq 0}$ defined on $X$, see \eqref{def:X}. As a consequence, we can use the results for strongly monotone semiflows since the space of continuous functions is endowed with an order defined in terms of a solid cone, i.e., with non-empty interior. See e.g., \cite{MR921986,MR1319817,MR1112067}. The authors in the aforementioned works proved that for a strongly monotone and continuous semiflow (or more generally for a strongly order-preserving and continuous semiflow) the long-term behavior of most of the trajectories is rather simple; mainly, the omega limit sets of most of the points are constituted by equilibria. We point out the fact that here the semiflow is assumed to be continuous since the dynamics of a discrete semiflow may be harder to describe, and the generic-behavior property can even fail; see \cite{MR1132766,MR1223450} and the references therein. Alternatively, other tools that have been successfully applied to describe the dynamics of discrete dynamical systems, parabolic autonomous and nonautonomous equations, and random dynamical systems arising from parabolic equations and delay equations are the principal Floquet bundles and the principal Lyapunov exponents. These notions extend naturally the properties of the principal eigenvalue of parabolic problems and therefore may serve as a generalization of the Krein-Rutman theorem; see \cite{MR2350061,MR3537364,MR2754337,MR4392478}.

\begin{proof}[Proof of Theorem \ref{th:main}]
Let $\{ u_\lambda \}_{\lambda\geq 0}$ be a family of initial distributions in $\L^1_+(0,+\infty)$ as stated in the theorem, and suppose that not all the solutions vanish. Let $\{b^\lambda\}_{\lambda \geq 0}$ denote the associated family of solutions to the integral equation \eqref{eq:int-proof}. Set $\{ \varphi_\lambda \}_{\lambda\geq 0}$ be the family in $C_+^0 = C_+^0([-a^*,0])$ defined by
\[
    \varphi_\lambda(t) = b^\lambda(t+a^*), \quad t\in [-a^*,0].
\]
Clearly, each function $\varphi_\lambda$ belongs to $X$ (see \eqref{def:X}) since $b^\lambda$ satisfies the integral equation in the right-half axis. In addition, this family verifies Assumption \ref{as-3} thanks to Assumption \ref{as-2}. Indeed, the continuity of the map $\lambda \mapsto \varphi_\lambda$ follows from the continuous dependence of the initial data (Proposition \ref{prop:basics}), and the monotonicity condition follows from the comparison principle. It just remains to prove that if $\lambda < \eta$ then $\varphi_\lambda < \varphi_\eta$. To do this, observe that for $t\in[0,a^*]$ we can rewrite the equation \eqref{eq:int-proof} for $b^\lambda$ as
\[
    b^\lambda(t) = f\left( \int_t^{a^*} \beta(a) e^{-\int_{a-t}^a \mu(l) \dd l} u_\lambda(a-t) \dd a + \int_0^t \beta(a) e^{-\int_0^a \mu(l) \dd l} b^\lambda(t-a) \dd a \right),
\]
and hence it is enough to show that
\[
    \int_t^{a^*} \beta(a) e^{-\int_{a-t}^a \mu(l) \dd l} [u_\eta(a-t) - u_\lambda(a-t)] \dd a >0,
\]
which is equivalent to
\begin{equation}\label{eq:positive}
    \int_0^{a^*} \beta(a+t)e^{-\int_a^{a+t} \mu(l) \dd l} [u_\eta(a) - u_\lambda(a)] \dd a >0,
\end{equation}
for some $t\in [0,a^*]$.

To prove this, recall from Assumption \ref{as-2} we know that the difference $u_\eta - u_\lambda$ is positive in some set of positive measure; thus, let us define $P=\{ a\in(0,a^*) \colon u_\eta(a) - u_\lambda(a) >0 \}$ and $M = \{ a\in (0,a^*) \colon \beta(a) > 0 \}$. Each of these sets have positive measure, and from the definition of $a^*$ we also have that $M\cap (t,a^*)$ has positive measure for every $t\in(0,a^*)$. Define the function $h(t) = \LL(P \cap (M-t))$, where $\LL$ denotes the Lebesgue measure. Subsequently, we obtain
\begin{align*}
    \int_0^{a^*} h(t) \dd t &= \int_0^{a^*} \int_0^{a^*} \1_{P\cap (M-t)}(s) \dd s \dd t \\
    &= \int_0^{a^*} \int_0^{a^*} \1_P(s) \1_M(s+t) \dd s \dd t\\
    &= \int_0^{a^*} \1_P(s) \int_0^{a^*} \1_M(s+t) \dd t \dd s \\
    &= \int_0^{a^*} \1_P(s) \LL(M \cap (s,a^*)) \dd s > 0,
\end{align*}
which implies that $h$ has to be positive in some subset of $a^*$ of positive measure, from where \eqref{eq:positive} follows. Furthermore, due to the uniqueness of the solution to the integral equation, it is clear that $\Phi(t,\varphi_\lambda) = b^\lambda_{t+a^*}$. Then, from Theorem \ref{th:threshold} we obtain the existence of a unique $\lambda^*>0$ such that 
\[
    \lim_{t\to \infty} b^\lambda(t) = \begin{cases}
        0, &\text{if}~ \lambda < \lambda^*, \\
        \kappa_2, &\text{if}~ \lambda> \lambda^*.
    \end{cases}
\]

Now, observe that from the formula of the semiflow $U$ we obtain
\begin{align*}
    \| U(t)u_\lambda \| &= \int_0^{+\infty} u(t,a) \dd a \\
    &= \int_t^{+\infty} e^{-\int_{a-t}^a \mu(l) \dd l} u_0(a-t) \dd a + \int_0^t e^{-\int_0^a \mu(l) \dd l } b(t-a) \dd a \\
    &\leq e^{-\unmu t} \|u_0\| + e^{-\unmu t}\int_0^{t} e^{\unmu a}b(a) \dd a,
\end{align*}
and the right-hand side goes to 0 as $t\to \infty$ if $\lambda < \lambda^*$. This shows that $U(t)u_\lambda \to 0$ as $t\to \infty$ if $\lambda< \lambda^*$ and a similar argument shows the convergence towards $\overline{\varphi}_2$ if $\lambda > \lambda^*$. 

On the other hand, if $\lambda = \lambda^*$ and $\varepsilon>0$ is given, by the last claim in Theorem \ref{th:threshold} there exists $\delta>0$ and  $T_\varepsilon>0$ with $b^{\lambda^*}(t) \leq \kappa_2 - \delta +\varepsilon$ for every $t\geq T_\varepsilon$. Then,
\[
    \| U(t)u_{\lambda^*} - \overline{\varphi}_2 \| = \int_0^{+\infty} \left| e^{-\int_{a}^{a+t} \mu(l) \dd l} u_0(a) - \kappa_2 e^{-\int_0^{a+t} \mu(l) \dd l} \right| \dd a + \int_0^t e^{-\int_0^a \mu(l) \dd l} |b(t-a) - \kappa_2| \dd a,
\]
where the first term in the right-hand side goes to 0 as $t\to \infty$, while the second term can be analyzed by splitting the integral in the sets $[0,t-T_\varepsilon]$ and $[t-T_\varepsilon, t]$, from where it follows
\[
    \liminf_{t\to \infty} \| U(t)u_{\lambda^*} - \overline{\varphi}_2 \| \geq \kappa_2 - \delta.
\]
The other relation is obtained analogously. This completes the proof of the theorem.
\end{proof}

Something remarkable from the works \cite{MR2608941,MR2754337,MR2169048} is that they also present a description of the asymptotic behavior of the threshold solutions for the treated models, besides the sharp transition between convergence to the trivial or nontrivial stable equilibria. Such a description is not included in our main result and is an open question for the moment. Notwithstanding, it is known that the age-structured model under consideration may produce oscillatory and periodic behavior when a nontrivial equilibrium loses stability; see, for instance, \cite{MR5010560,MR1822211,MR4683856,MR2559965}. In particular, in \cite{MR1719128} the authors proved that the global attractor for the system may be visualized as a solid spindle that is divided by a disk whose boundary is a periodic orbit. A center manifold theorem played a key role in this description, and therefore we hope that the tools developed by Ducrot, Magal and Ruan may help to prove the periodic behavior effectively for this system, e.g., \cite{MR3988617,MR2559965}. 

\section{A non-compactly supported case}\label{sec:non_compact}

This section is devoted to the study of a particular case in which the birth rate $\beta$ has non-compact support, which implies $a^* = + \infty$. More precisely, throughout this section it is assumed that Assumption \ref{as:4} holds. The main difficulty with the case $a^* = + \infty$ is that the effect produced by the initial population is always carried over all the time as it passes; see the first term on the right-hand side of \eqref{eq:int_1}. In contrast with the case treated before, the solution of the integral equation given by the Volterra formulation of the problem is not completely determined by the previous information over a compact interval. This change makes the analysis more involved, and the strongly order-preserving property actually fails, see the Appendix. Nevertheless, the eventually constant behavior for the birth and date rates, stated in Assumption \ref{as:4}, allows us to translate the problem for the semiflow generated by the age-structured model \eqref{GM-model} into one for a semiflow defined on a subset of the space of continuous functions over a compact interval. To accomplish this, the idea is to introduce a new variable that tracks over time the mass produced by the individuals of a large enough age. 

Let us start with a quick remark that follows from the new hypothesis. Due to Assumption \ref{as:4}, the normalization hypothesis made over $\beta$ and $\mu$ reads
\begin{align}
    1 &= \int_0^{+\infty} \beta(a)e^{-\int_0^a \mu(l) \dd l} \dd a \nonumber \\
    &= \int_0^{a_0} \beta(a)e^{-\int_0^a \mu(l) \dd l} \dd a + \beta_\infty \int_{a_0}^{+\infty} e^{-\int_0^{a_0} \mu(l) \dd l}e^{-\int_{a_0}^a \mu(l) \dd l} \dd a \nonumber \\
    &= \int_0^{a_0} \beta(a)e^{-\int_0^a \mu(l) \dd l} \dd a + \beta_\infty e^{-\int_0^{a_0}\mu(l) \dd l} \int_{a_0}^{+\infty} e^{-\mu_\infty (a-a_0)} \dd a \nonumber \\
    &= \int_0^{a_0} \beta(a)e^{-\int_0^a \mu(l) \dd l} \dd a + \frac{\beta_\infty}{\mu_\infty} e^{-\int_0^{a_0} \mu(l) \dd l}. \label{eq:norm_2}
\end{align}

Now, after integrating along the characteristics, we have the following representation for the solution $u$ of the Gurtin-MacCamy model
\[
u(t,a) = \begin{cases}
    e^{-\int_{a-t}^a \mu(l) \dd l} u_0(a-t), & a\geq t, \\
    e^{-\int_0^a \mu(l) \dd l a} b(t-a), &a\leq t,
\end{cases}
\]
where $b\colon \R_+ \to \R$ is the unique continuous solution to
\[
    b(t) = f\left( \int_t^{+\infty} \beta(a)e^{-\int_{a-t}^a \mu(l) \dd l} u_0(a-t) \dd a + \int_0^t \beta(a) e^{-\int_0^a \mu(l) \dd l} b(t-a) \dd a \right), \quad t\geq 0,
\]
and observe that $b(t) = f\left( \int_0^{+\infty} \beta(a) u(t,a) \dd a \right)$. Next, by integrating the differential equation in \eqref{GM-model} over $[a_0,+\infty)$ follows (at least formally)
\begin{equation}\label{eq:int_I}
    \partial_t \int_{a_0}^{+\infty} u(t,a) \dd a = u(t,a_0) - \mu_\infty\int_{a_0}^{+\infty} u(t,a) \dd a.
\end{equation}
On the other hand, if $t\geq a_0$ and $a\geq a_0$, integrating along the characteristics in $[t-a_0,t]$ we obtain
\[
    u(t,a_0) = e^{-\int_0^{a_0} \mu(l) \dd l} u(t-a_0,0) = e^{-\int_0^{a_0} \mu(l) \dd l} f \left( \int_0^{+\infty} \beta(\sigma) u(t-a_0, \sigma) \dd \sigma \right).
\]
Therefore, letting $I(t) = \int_{a_0}^{+\infty} u(t,a) \dd a$, \eqref{eq:int_I} turns out to be
\begin{align*}
    I'(t) &= e^{-\int_0^{a_0} \mu(l) \dd l} b(t-a_0) - \mu_\infty I(t);
\end{align*}
whereas the integral equation for $b$ for $t\geq a_0$ can be rewritten as
\[
    b(t) = f\left( \int_0^{a_0} \beta(a) e^{-\int_0^a\mu(l) \dd l} b(t-a) \dd a + \beta_\infty I(t) \right).
\]
Therefore, under Assumption \ref{as:4}, it is apparent that there is a close relation between the dynamics of the Gurtin-MacCamy model and the following system which couples a differential equation and a nonlinear integral equation:
\begin{equation}\label{delay_sys}
    \begin{dcases}
    I'(t) = -\mu_\infty I(t) + e^{-\int_0^{a_0} \mu(l) \dd l}b(t-a_0), & t\geq 0, \\
    b(t) = f\left( \int_0^{a_0} \beta(a) e^{-\int_0^a \mu(l) \dd l} b(t-a) \dd a + \beta_\infty I(t) \right), & t\geq 0,
    \end{dcases}
\end{equation}
equipped with some initial condition $(I(0),b_0) = (\alpha,\varphi)\in \R_+ \times C_+^0([-a_0,0])$, where we have used the notation $b_t(\theta) = b(t+\theta)$, for $t\geq 0$ and $\theta\in [-a_0, 0]$. Hereinafter, for an element $\varphi\in C_+^0([-a_0,0])$ we denote by $\|\varphi\|$ the usual sup-norm of this element. System \ref{delay_sys} has exactly three equilibria $(\overline{I},\overline{b})$ (see \eqref{eq:norm_2}) given by
\[
    (0,0), \quad \left( \frac{c_0}{\mu_\infty} \kappa_1, \kappa_1 \right), \quad \text{and} \quad \left( \frac{c_0}{\mu_\infty} \kappa_2, \kappa_2 \right).
\]

Note that the first equation in \ref{delay_sys} can be solved directly for $I$ as
\[
    I(t) = e^{-\mu_\infty t} I(0) + c_0 \int_0^t e^{-\mu_\infty(t-s)} b(s-a_0) \dd a,
\]
with $c_0 := e^{-\int_0^{a_0} \mu(l) \dd l}$. Consequently, for a specified initial data $(\alpha,\varphi)\in \R_+ \times C_+^0([-a_0,0])$, we can determine the values of the solution $I(t)$ on the interval $[0,a_0]$. Then, applying a standard contraction argument, the existence of a solution for the second equation in \ref{delay_sys} on the interval $[0,a_0]$ follows. Subsequently, repeating this process yields the existence of a unique global continuous solution for the system. We can also obtain the continuous dependence of the solution on the initial data $(\alpha,\varphi)$ through an induction argument over the intervals of the form $[0,na_0]$, $n\in\N$ using the same techniques as in the preceding case; cf. Proposition \ref{prop:basics}. Let us summarize these properties in the following statement.

\begin{proposition}
    For every initial data $(\alpha,\varphi)\in \R_+ \times C_+^0([-a_0,0])$, the system \eqref{delay_sys} has a unique global continuous and nonnegative solution $(I(t),b(t))$ such that $(I(0),b_0)=(\alpha,\varphi)$. In addition, there is a continuous dependence on the initial data in the following sense: given $\varepsilon>0$, $(\alpha,\varphi)\in \R_+ \times C_+^0([-a_0,0])$ and $T>0$, then there is $\delta>0$ such that for every $(\tilde{\alpha},\tilde{\varphi})\in \R_+ \times C_+^0([-a_0,0])$ with $|\alpha - \tilde{\alpha}| + \|\varphi - \tilde{\varphi}\| < \delta$ then it holds $|I(t) - \tilde{I}(t)| + |b(t) - \tilde{b}(t)| < \varepsilon$ for all $t\in[0,T]$.
\end{proposition}

Concerning the boundedness properties, we have the following statement (c.f. Theorem \ref{th:bound}).

\begin{theorem}\label{conv_delay}
    The solutions of the system \eqref{delay_sys} are uniformly ultimately bounded. Indeed, let
    \[
        (I^*, b^*) = \left( \limsup_{t\to \infty} I(t), \limsup_{t\to \infty} b(t)  \right), \quad\text{and}\quad (I_*,b_*) = \left( \liminf_{t\to \infty} I(t), \liminf_{t\to \infty} b(t) \right),
    \]
    then
    \[
        b^*\in \{0\} \cup [\kappa_1, \kappa_2], \quad b_*\in [0,\kappa_1] \cup \{ \kappa_2 \}, \quad \text{and}\quad [I_*, I^*] \subseteq \left[ \frac{c_0}{\mu_\infty} b_*, \frac{c_0}{\mu_\infty} b^*\right],
    \]
    where $c_0 := e^{-\int_0^{a_0} \mu(l) \dd l}$. Moreover, there are constants $C_1,C_2>0$ depending only on $f$, $\beta$ and $\mu$ such that.  
    \[
        b(t) \leq f(C_1 + C_2 \| (\alpha, \varphi) \|),
    \]
    where $\|( \alpha, \varphi )\|$ denotes the usual product norm in $\R_+ \times C_+^0([-a_0,0])$.
\end{theorem}
\begin{proof}
Let $(I(t),b(t))$ be a solution to \eqref{delay_sys} associated to the initial data $(\alpha, \varphi)$. Let
\begin{align*}
    B(t) &:= \int_0^{a_0} \beta(a) e^{-\int_0^a \mu(l) \dd l} b(t-a) \dd a + \beta_\infty I(t)\\
    &= \int_0^t \tilde{\beta}(a) e^{-\int_0^a \mu(l) \dd l} b(t-a) \dd a + \sigma(t) + \beta_\infty I(t),
\end{align*}
for $t\geq 0$, where
\begin{equation}\label{eq:new_par}
    \tilde{\beta}(a) = \begin{cases}
        \beta(a), & a\in(a,a_0), \\
        0, & a>a_0,
    \end{cases}
    \quad \text{and}\quad 
    \sigma(t) = \begin{cases}
        \int_t^{a_0} \beta(a) e^{-\int_0^a \mu(l) \dd l} \varphi(t-a) \dd a, & t\in[0,a_0], \\
        0, & t>a_0.
    \end{cases}
\end{equation}
This definition is such that $b(t) = f(B(t))$ for $t\geq 0$, and $B$ solves the following equation
\begin{equation}
    B(t) = \int_0^t \tilde{\beta}(a) e^{-\int_0^a \mu(l) \dd l} f(B(t-a)) \dd a + \sigma(t) + \beta_\infty I(t), \quad t\geq 0. \label{eq:B_1}
\end{equation}
Alternatively, the equation for $I$ can be solved explicitly as
\begin{equation}\label{eq:sol_I}
    I(t) = \alpha e^{-\mu_\infty t} + c_0 \int_0^t e^{-\mu_\infty (t-s)} b(s-a_0) \dd s,
\end{equation}
with $c_0 = e^{-\int_0^{a_0} \mu(l) \dd l}$ as in the statement of the Theorem. Hence, from the definition of $B$ and \eqref{eq:sol_I} we have for $t\geq a_0$
\begin{align}
    B(t) &= \int_0^{a_0} \beta(a)e^{-\int_0^a \mu(l) \dd l} f(B(t-a)) \dd a + \beta_\infty \left[ \alpha e^{-\mu_\infty t} + c_0\int_0^t e^{-\mu_\infty (t-s)} b(s-a_0) \dd s \right] \nonumber \\
    &= \int_0^{a_0} \beta(a)e^{-\int_0^a \mu(l) \dd l} f(B(t-a)) \dd a + \beta_\infty \left[ \alpha e^{-\mu_\infty t} + c_0 \int_{a_0}^{t+a_0} e^{-\mu_\infty (s-a_0)} b(t-s) \dd s \right] \nonumber \\
    &= \int_0^{a_0} \beta(a)e^{-\int_0^a \mu(l) \dd l} f(B(t-a)) \dd a + \beta_\infty \left[ \alpha e^{-\mu_\infty t} + c_0 \int_{a_0}^{t+a_0} e^{-\int_{a_0}^s \mu(l) \dd l} b(t-s) \dd s \right] \nonumber \\
    &= \int_0^{a_0} \beta(a)e^{-\int_0^a \mu(l) \dd l} f(B(t-a)) \dd a + \alpha \beta_\infty e^{-\mu_\infty t} + \int_{a_0}^{t+a_0} \beta(a) e^{-\int_0^a \mu(l) \dd l} b(t-a) \dd a \nonumber \\
    &= \int_0^t \beta(a) e^{-\int_0^a \mu(l) \dd l} f(B(t-a)) \dd a + \alpha\beta_\infty e^{-\mu_\infty t} + \int_t^{t+a_0} \beta(a) e^{-\int_0^a \mu(l) \dd l} \varphi(t-a) \dd a. \label{eq:B_2}
\end{align}

Let $\rho\in (0,1)$, $M>0$ and $K$ be as in the proof of Theorem \ref{th:bound}. Let also $T\geq a_0$ be given. From \eqref{eq:B_1} and \eqref{eq:sol_I}, the following estimates follow for $t\in[0,a_0]$
\begin{align*}
    B(t) &\leq \int_0^t \beta(a) e^{-\int_0^{a} \mu(l) \dd l} f(B(t-a)) \dd a + \ovbeta \| \varphi \| \int_t^{a_0} e^{-\unmu a} \dd a + \beta_\infty \left[ \alpha e^{-\mu_\infty t} + c_0 \|\varphi\| e^{-\mu_\infty t} \int_0^t e^{\mu_\infty s} \dd s \right] \\
    &\leq \ovbeta\|\varphi\| a_0 + \alpha \beta_\infty + c_0 \beta_\infty \| \varphi \| t + \left( \int_{[0,t]\cap K^c} + \int_{[0,t] \cap K} \right) \beta(t-a) e^{-\int_0^{t-a} \mu(l) \dd l} f(B(a)) \dd a \\
    &\leq a_0 \ovbeta \|\varphi\| + \alpha\beta_\infty + a_0c_0\beta_\infty \| \varphi \| + \left( f(M) + \rho \max_{s\in [0,T]} B(s) \right) \int_0^t \beta(t-a) e^{-\int_0^{t-a} \mu(l) \dd l} \dd a \\
    &\leq a_0 \ovbeta \|\varphi\| + \alpha\beta_\infty + a_0c_0\beta_\infty \| \varphi \| + f(M) + \rho \max_{s\in[0,T]} B(s).
\end{align*}
Analogously, from \eqref{eq:B_2}, the following estimates hold for $t\in[a_0,T]$
\begin{align*}
    B(t) &\leq \alpha \beta_\infty e^{-\mu_\infty t} + \ovbeta \|\varphi\| \int_t^{t+a_0} e^{-\unmu a} \dd a + \int_0^t \beta(a) e^{-\int_0^a \mu(l) \dd l} f(B(t-a)) \dd a \\
    &\leq \alpha \beta_\infty + a_0 \ovbeta \|\varphi\| + f(M) + \rho \max_{s\in [0,T]} B(s),
\end{align*}
and therefore we infer that
\[
    \max_{t\in [0,T]} B(t) \leq \frac{f(M) + \alpha \beta_\infty + (a_0\ovbeta + a_0c_0\beta_\infty ) \| \varphi \|}{1-\rho},
\]
which shows the boundedness of $B$ since the right-hand side is independent of $T$. Consequently, $b$ and $I$ are also bounded, and the bound for $b$ follows from that one for $B$.

In addition, using \eqref{eq:B_2} and the same arguments as in the proof of Theorem \ref{th:bound}, we deduce that
\[
    \limsup_{t\to \infty} B(t) \in \{ 0 \} \cup [\kappa_1, \kappa_2], \quad\text{and}\quad \liminf_{t\to\infty} B(t) \in[0,\kappa_1] \cup \{ \kappa_2 \},
\]
and thanks to the assumptions made on $f$ the same conclusion holds for $b$. For the conclusion about the asymptotic behavior of $I$, let $(t_n)_{n\in \N}$ be a sequence of real numbers such that $t_n\to \infty$ and $I(t_n) \to I^*$. Let also $\varepsilon>0$ be given and $T_\varepsilon>0$ such that $b(t) \leq b^* + \varepsilon$ for $t\geq T_\varepsilon$. Then, starting from its representation given in terms of $b$ \eqref{eq:sol_I}, we get for $n$ large enough that
\begin{align*}
    I(t_n) &\leq \alpha e^{-\mu_\infty t_n} + c_0 \left( \int_0^{T_\varepsilon + a_0}e^{-\mu_\infty(t_n-s)} b(s-a_0) \dd s + (b^*+\varepsilon) \int_{T_\varepsilon+a_0}^{t_n} e^{-\mu_\infty(t_n-s)} \dd s \right) \\
    &= \alpha e^{-\mu_\infty t_n} + c_0 \left( \int_0^{T_\varepsilon + a_0}e^{-\mu_\infty(t_n-s)} b(s-a_0) \dd s + (b^*+\varepsilon) \int_0^{t_n - T_\varepsilon - a_0} e^{-\mu_\infty s} \dd s \right).
\end{align*}
Thus, taking $n\to \infty$ it follows that $I^* \leq \frac{c_0}{\mu_\infty} (b^*+\varepsilon)$, and since $\varepsilon$ is arbitrary, we infer that $I^* \leq \frac{c_0}{\mu_\infty} b^*$. The estimate for $I_*$ is proved in the same fashion (cf. Theorem \ref{th:bound}).
\end{proof}

The proof of the above theorem shows that it is possible to rewrite the equation for $b$ as a Volterra equation, mainly
\begin{equation}\label{eq:alter_b}
    b(t) = f\left( \int_0^t \tilde{\beta}(a) e^{-\int_0^a \mu(l) \dd l} b(t-a) \dd a + \sigma(t) + \beta_\infty I(t) \right),
\end{equation}
with $\tilde{\beta}$ and $\sigma$ as in \eqref{eq:new_par}. Using this form of the equation, and since the solution of the system \eqref{delay_sys} is obtained by a contraction mapping argument, it is possible to extend the results of the last paragraph of Section \ref{sec:prelim}. We summarize those results as follows.

\begin{proposition}\label{comparison}
If $(\alpha, \varphi),(\tilde{\alpha},\tilde{\varphi})\in \R_+ \times C_+^0([-a_0,0])$ with $(\alpha, \varphi) \leq (\tilde{\alpha}, \tilde{\varphi})$ are given (where the order relation is understood component-wise), then the associated solutions of the system \eqref{delay_sys} denoted by $(I(t),b(t))$ and $(\tilde{I}(t), \tilde{b}(t))$ verify
\[
    (I(t),b(t)) \leq (\tilde{I}(t), \tilde{b}(t)), \quad t\geq 0.
\]
Besides, let $(\alpha, \varphi)\in \R_+ \times C_+^0([-a_0,0])$ be given such that $(\alpha, \varphi) \geq(\leq) \left( \frac{c_0}{\mu_\infty} \kappa_1, \kappa_1 \right)$ but $(\alpha, \varphi) \neq \left( \frac{c_0}{\mu_\infty} \kappa_1, \kappa_1 \right)$. Then the associated solution $(I(t),b(t))$ verifies
\[
    \lim_{t\to \infty}(I(t),b(t)) = \left( \frac{c_0}{\mu_\infty} \kappa_2, \kappa_2 \right) \big( (0,0)\big).
\]
\end{proposition}
\begin{proof}
Let $(\alpha,\varphi),(\tilde{\alpha},\tilde{\varphi})\in \R_+ \times C_+^0([-a_0,0])$ be given, with $(\alpha,\varphi) \leq (\tilde{\alpha},\tilde{\varphi})$, and let $(I(t),b(t))$ and $(\tilde{I}(t), \tilde{b}(t))$ be the associated solutions. In addition, consider the operator $\AA_\varphi \colon C_+^0([0,a_0]) \to C_+^0([0,a_0])$ defined by
\[
    \AA_\varphi(\psi)(t) = f\left( \int_0^t \beta(a) e^{-\int_0^a \mu(l) \dd l} \psi(t-a) \dd a + \sigma(t) + \beta_\infty I(t) \right),
\]
with $\sigma$ as in \eqref{eq:new_par}. Define similarly $\AA_{\tilde{\varphi}} \colon C_+^0([0,a_0]) \to C_+^0([0,a_0])$. Observe that the solutions $b$ and $\tilde{b}$ are fixed points of the operators $\AA_\varphi$ and $\AA_{\tilde{\varphi}}$, respectively. Then, from the explicit representation for $I$ and $\tilde{I}$ on $[0,a_0]$ we see
\begin{align*}
    I(t) &= \alpha e^{-\mu_\infty t} + c_0 \int_0^t e^{-\mu_\infty(t-s)} \varphi(s-a_0) \dd s \\
    &\leq \tilde{\alpha}e^{-\mu_\infty t} + c_0 \int_0^t e^{-\mu_\infty(t-s)} \tilde{\varphi}(s-a_0) \dd s = \tilde{I}(t),
\end{align*}
for $t\in [0,a_0]$. Therefore, due to the monotonicity of $f$ it follows
\begin{align*}
    b(t) = \AA_\varphi(b)(t) \leq \AA_{\tilde{\varphi}}(b)(t),
\end{align*}
which, jointly with an induction argument, implies $b \leq \AA_{\tilde{\varphi}}(b)$. Since $\tilde{b}$ is constructed by a contraction mapping theorem this implies that $b \leq \tilde{b}$ over $[0,a_0]$. The result for all $t\geq 0$ is obtained by induction.

The last claim in the Theorem follows as in the proof of Corollary \ref{cor:conv}, thanks to the comparison principle just stated.
\end{proof}

\subsection{The dynamical system formulation}

Just as in Section \ref{th:threshold}, for a given initial data $(\alpha, \varphi)\in \R_+ \times C_+^0([-a_0,0])$, the associated solution for the second equation in \eqref{delay_sys} may have a discontinuity at $t=0$ provided
\[
    \varphi(0) \neq f\left( \int_0^{a_0} \beta(a) e^{-\int_0^a\mu(l) \dd l} \varphi(-a) \dd a + \alpha \beta_\infty \right),
\]
while the solution $I$ is a continuous function on $[-a_0,+\infty)$. Consequently, for the purpose of setting the problem in a dynamical system framework, let us consider the following closed subset $Y\subseteq \R_+ \times C_+^0([-a_0,0])$
\begin{equation}\label{def:Y}
    Y:= \left\{ (\alpha, \varphi)\in \R_+ \times C_+^0([-a_0,0]) \colon f\left( \int_0^{a_0} \beta(a) e^{-\int_0^a\mu(l) \dd l} \varphi(-a) \dd a + \alpha \beta_\infty \right) = \varphi(0)  \right\},
\end{equation}
which contains all the initial conditions giving rise to a continuous solution $b$ for $[-a_0,+\infty)$. Thus, we can consider the continuous semiflow generated by the system \eqref{delay_sys} on $Y$, say, $\Psi \colon \R_+ \times Y \to Y$ and given by $\Psi(t,(\alpha, \varphi)) = (I(t),b_t)$, where $I$ and $b$ are the associated solutions to the system. The equilibrium points for this semiflow are the constant solutions to \eqref{delay_sys}, that is $(0,0)$, $\left( \frac{c_0}{\mu_\infty} \kappa_1, \kappa_1 \right)$ and $\left( \frac{c_0}{\mu_\infty} \kappa_2, \kappa_2 \right)$. Because of the comparison principle established in the previous paragraph, the semiflow $\Psi$ is monotone and, as we shall see, is eventually strongly monotone and eventually compact. Indeed, the proof of the next proposition follows the same lines as in Proposition \ref{compact_semifl} by using the comparison principle and the alternative equation for $b$ \eqref{eq:alter_b}.

\begin{proposition}\label{compac_delay}
The semiflow $\Psi$ is eventually strongly monotone and eventually compact. That is, it satisfies the following statements.
\begin{enumerate}[label=(\roman*)]
    \item Let $(\alpha, \varphi) \leq (\tilde{\alpha},\tilde{\varphi})$ be two elements in $Y$ given with $(\alpha, \varphi) \neq (\tilde{\alpha},\tilde{\varphi })$. Then there is some $t_0>0$ such that $\Psi(t_0,(\alpha, \varphi)) \ll \Psi(t_0,(\tilde{\alpha},\tilde{\varphi}))$. As a matter of fact, $\Psi(t, (\alpha, \varphi)) \ll \Psi(t,(\tilde{\alpha},\tilde{\varphi}))$ for all $t\geq t_0$.
    \item For every bounded subset $B\subseteq Y$, $\Psi(t,B)$ has compact closure for $t\geq a_0$.
\end{enumerate}
Furthermore, there exists $\AA\subset Y$ a compact attractor of bounded sets for the semiflow $\Psi$.
\end{proposition}
\begin{proof}
For the first part, let $(\alpha, \varphi) \leq (\tilde{\alpha},\tilde{\varphi})$ be two elements in $Y$ given with $(\alpha, \varphi) \neq (\tilde{\alpha},\tilde{\varphi })$; and net $(I(t),b(t))$ and $(\tilde{I}(t), \tilde{b}(t))$ be the associated solutions.  Assume first that $\alpha < \tilde{\alpha}$. Then, we can see directly from the explicit formula for $I$ and $\tilde{I}$ such that $I(t) < \tilde{I}(t)$ for all $t\geq 0$ because of the comparison principle stated in Proposition \ref{comparison}. Therefore, since $f$ is strictly monotone, we deduce from the second equation in the system \eqref{delay_sys} that $b(t) < \tilde{b}(t)$ for all $t\geq 0$. Therefore, $\Psi(t,(\alpha,\varphi)) \ll \Psi(t,(\tilde{\alpha},\tilde{\varphi}))$ for all $t\geq a_0$.

On the other hand, suppose that $\varphi < \tilde{\varphi}$. As the computation of the values $I(t)$ and $\tilde{I}(t)$ need all the history of $b$ and $b(t)$ before $t-a_0$ as can be seen in the representation of these solutions, then $I(t) < \tilde{I}(t)$ for all $t\geq a_0$. Then, just as in the paragraph above, it follows that $b(t) < \tilde{b}(t)$ for all $t\geq a_0$, from where the first part of the proposition is proved.

The following statements in the proposition are proved in the same way as was done in Proposition \ref{compact_semifl} and Corollary \ref{attractor}.
\end{proof}

Recall that in the compactly supported case we defined the separatrix in terms of oscillations of the solution to the integral equation around the intermediate equilibria, or more precisely, in terms of intersections with it. For the non-compactly supported case, it is not possible to characterize the separatrix similarly. Mainly, note that the function $I$ tracks all the past values of $b$ on the interval $[-a_0,t-a_0]$, and therefore the knowledge about the eventual behavior of $b$ (namely, if its position is above or below the unstable equilibria for large times) does not provide any further information about the position of $I$ around the point $\frac{c_0}{\mu_\infty} \kappa_1$ $\kappa_1/\mu$. Hence, to describe the separatrix, we will use the fact that the system possesses exactly three equilibria, and the middle one is unstable. To this aim, let $A_0$ and $A_{\kappa_2}$ be the basins of attraction for the equilibrium points $(0,0)$ and $(\kappa_2/\mu, \kappa_2)$, namely
\begin{align*}
    A_0 &= \{ (\alpha, \varphi)\in Y \colon \Psi(t,(\alpha, \varphi)) \to (0,0)~\text{as}~ t\to\infty \}, \\
    A_{\kappa_2} &= \left\{ (\alpha, \varphi)\in Y \colon \Psi(t,(\alpha, \varphi)) \to \left( \frac{c_0}{\mu_\infty}\kappa_2, \kappa_2 \right) ~\text{as} ~t\to\infty \right\}.
\end{align*}
These sets are open thanks to the continuity of the semiflow and the local stability of the equilibria given in Proposition \ref{comparison}.

\begin{proposition}
The set $Y \setminus (A_0 \cup A_{\kappa_2})$ is totally unordered.
\end{proposition}
\begin{proof}
The proof follows the same arguments given for Proposition \ref{unord_S}. Let $(\alpha, \varphi)\leq (\tilde{\alpha}, \tilde{\varphi})$ be two distinct elements in $Y\setminus (A_0 \cup A_{\kappa_2})$. As $\Psi$ is an eventually strongly monotone semiflow there is a sufficiently large $t_0>0$ such that $\Psi(t, (\alpha, \varphi)) \ll \Psi(t, (\tilde{\alpha}, \tilde{\varphi}))$ for each $t\geq t_0$. Then, as the quasiconvergence is generic (see Theorem 4.3 in \cite{MR1319817}) we may take $(\eta,\rho) \ll (\tilde{\eta},\tilde{\rho})$ two elements in $Y$ such that $\omega((\eta,\rho)), \omega((\tilde{\eta},\tilde{\rho})) \subseteq \left\{ (0,0), \left( \frac{c_0}{\mu_\infty}\kappa_1, \kappa_1 \right) , \left( \frac{c_0}{\mu_\infty}\kappa_2, \kappa_2 \right) \right\}$ and $\Psi(t_0,(\alpha, \varphi)) \ll (\eta,\rho) \ll (\tilde{\eta},\tilde{\rho}) \ll \Psi(t_0, (\tilde{\varphi}, \tilde{\psi})$. Since the initial functions $(\alpha, \varphi)$ and $(\tilde{\alpha}, \tilde{\varphi})$ do not converge to any of the stable equilibria, we deduce that
\begin{equation}\label{eq:conv_psi}
    \Psi(t,(\eta,\rho))\to \left( \frac{c_0}{\mu_\infty}\kappa_1, \kappa_1 \right) \quad\text{and} \quad \Psi(t,(\tilde{\eta}, \tilde{\rho})) \to \left( \frac{c_0}{\mu_\infty}\kappa_1, \kappa_1 \right),
\end{equation}
as $t\to\infty$. Let $(I(t),b(t))$ and $(\tilde{I}(t), \tilde{b}(t))$ be the solutions of \eqref{delay_sys} with initial data $(\eta,\rho)$ and $(\tilde{\eta},\tilde{\rho})$, respectively, and let $\xi(t) = \tilde{b}(t) - b(t)$. Moreover, set
\[
    B(t) = \int_0^{a_0} \beta(a) e^{-\int_0^a \mu(l) \dd l} b(t-a) \dd a + \beta_\infty I(t), \quad t\geq 0,
\]
and an analogous definition for $\tilde{B}$. Then, the respective equations satisfied by $b$ and $\tilde{b}$ can be rewritten as $b(t) = f(B(t))$ and $\tilde{b}(t) = f(\tilde{B}(t))$, respectively. Hence, $\xi(t) = g(t)(\tilde{B}(t) - B(t))$, where
\[
    g(t) = \int_0^1 f' (s \tilde{B}(t) + (1-s) B(t)) \dd s,
\]
and due to \eqref{eq:conv_psi} we infer that $g(t) \to f'(\kappa_1)$ as $t\to \infty$.

Let $\delta \in (1,f'(\kappa_1))$, and note that there exists $T>0$ such that $g(t) \geq \delta$ for all $t\geq T$. Moreover, as in the proof of Theorem \ref{conv_delay} (see \eqref{eq:B_2}) we deduce
\begin{align*}
    \tilde{B}(t) - B(t) &= \int_0^{a_0} \beta(a) e^{-\int_0^a \mu(l) \dd l} \xi(t-a) \dd a + \beta_\infty (\tilde{I}(t) - I(t)) \\
    &= \int_0^{a_0} \beta(a) e^{-\int_0^a \mu(l) \dd l} \xi(t-a) \dd a + \beta_\infty \left[ (\tilde{\eta} - \eta) e^{-\mu_\infty t} + c_0 \int_0^t e^{-\mu_\infty(t-s)} \xi(s-a_0) \dd s \right] \\
    &= \int_0^t \beta(a) e^{-\int_0^a \mu(l) \dd l} \xi(t-a) \dd a + \beta_\infty(\tilde{\eta} - \eta) e^{-\mu_\infty t} \\
    & \qquad + \int_t^{t+a_0} \beta(a) e^{-\int_0^a \mu(l) \dd l} (\tilde{\varphi}(t-a) - \varphi(t-a)) \dd a \\
    &\geq \int_0^t \beta(a) e^{-\int_0^a \mu(l) \dd l} \xi(t-a) \dd a
\end{align*}
Then, as $\xi(t)>0$ for all $t\geq 0$ (see the proof of Proposition \ref{compac_delay}), and it vanishes at infinity, we can take $(t_n)_{n\in \N}$ a sequence of positive real numbers such that $t_n\to \infty$ and $\xi(t_n) = \min_{t\in[0,t_n]} \xi(t)$ for all $n\in \N$. Using all these properties we obtain, for $n$ large enough that
\[
    \xi(t_n) \geq \delta \xi(t_n) \int_0^{t_n} \beta(a) e^{-\int_0^a \mu(l) \dd l} \dd a,
\]
which leads us to the contradiction $\delta \leq 1$ after passing to the limit as $n\to \infty$ because of the normalization hypothesis \eqref{eq:norm_2}. This concludes the proof of the theorem.
\end{proof}

The latter result reveals that for the case we are considering, the set of initial data that produces a solution not converging to any of the stable equilibria is a nowhere dense set. We emphasize that the proof mainly relies on the structure of having exactly three equilibria for the semiflow. Nevertheless, what can be expected is that if the number of equilibria increases, then we may consider the set of initial data whose solutions do not converge to any of the stable ones, and therefore inferring this set is totally unordered as well.

From this result can be deduced a sharp threshold property for the system \eqref{delay_sys} as was done in the last paragraph of Section \ref{sec:threshold}. Let us summarize this result in the next assertion.

\begin{theorem}\label{th:threshold_2}
Let $\{ (\alpha_\lambda, \varphi_\lambda) \}_{\lambda\geq 0}$ be a family of initial functions in $Y$ such that the map $\lambda \to (\alpha_\lambda, \varphi_\lambda)$ is continuous from $\R_+$ into $Y$ and strictly monotone, i.e., if $\lambda <\eta$ then $(\alpha_\lambda, \varphi_\lambda) < (\alpha_\eta, \varphi_\eta)$. Assume as well that $(\alpha,\varphi_0) = (0,0)$, and denote by $(I^\lambda(t),b^\lambda(t))$ the associated solution of the system \eqref{delay_sys}. Then, every solution $I^\lambda$ and $b^\lambda$ vanishes at infinity, or there exists exactly one $\lambda^*>0$ such that
\begin{itemize}
    \item for all $\lambda < \lambda^*$, $(I^\lambda(t), b^\lambda(t)) \to (0,0)$ ad $t\to \infty$ is valid;
    \item and for all $\lambda > \lambda^*$, $(I^\lambda(t), b^\lambda(t)) \to \left( \frac{c_0}{\mu_\infty} \kappa_2, \kappa_2 \right)$ as $t\to \infty$ is valid.
\end{itemize}
Furthermore, in the latter case, there exists $\delta>0$ with
\[
    \delta \leq \frac{\mu_\infty}{c_0} \liminf_{t\to \infty} I^{\lambda^*}(t) \leq \frac{\mu_\infty}{c_0} \limsup_{t\to \infty} I^{\lambda^*}(t) \leq \kappa_2 - \delta,
\]
and
\[    
    \delta \leq \liminf_{t\to \infty} b^{\lambda^*}(t) \leq \limsup_{t\to \infty} b^{\lambda^*}(t) \leq \kappa_2- \delta.
\]
\end{theorem}

\subsection{Return to the age-structured model}

In this last paragraph we make explicit the relation between the age-structured model \eqref{GM-model} and the system \eqref{delay_sys}. The deduction of the latter one was done by some formal computations made in the Gurtin-MacCamy model, as the differentiation under the integral sign or even assuming that the solution given by the method of the characteristics is a classical solution of the model, that is, a solution that is indeed differentiable in both variables. Notwithstanding, the relation between these two formulations is easy to state, departing from the representation formulas given for the solutions and functions involved. Recall that the solution of the age-structured model given by the method of the characteristics is
\[
    u(t,a) = \begin{cases}
        e^{-\int_{a-t}^a \mu(l) \dd l} u_0(a-t), & a\geq t, \\
        e^{-\int_0^a \mu(l) \dd l} b(t-a), & t\geq a,
    \end{cases}
\]
where $b\colon \R_+ \to \R$ is the solution to the following integral equation
\[
    b(t) = f\left( \int_t^{+\infty} \beta(a)e^{-\int_{a-t}^a \mu(l) \dd l} u_0(a-t) \dd a + \int_0^t \beta(a)e^{-\int_0^a \mu(l) \dd l} b(t-a) \dd a \right).
\]
Let $I(t) = \int_{a_0}^{+\infty} u(t,a) \dd a$. Then, for $t\geq a_0$ we obtain
\begin{align}
    I(t) &= \int_t^{+\infty} e^{-\int_{a-t}^a \mu(l) \dd l} u_0(a-t) \dd a + \int_{a_0}^t e^{-\int_0^a\mu(l) \dd l} b(t-a) \dd a \label{expr_I} \\
    &= \int_0^{+\infty} e^{-\int_a^{a+t} \mu(l) \dd l} u_0(a) \dd a + e^{-\int_0^{a_0} \mu(l) \dd l} \int_{a_0}^t e^{-\mu_\infty (a-a_0)} b(t-a) \dd a \nonumber \\
    &= \int_0^{a_0} e^{-\int_a^{a_0} \mu(l) \dd l} e^{-\mu_\infty(t+a-a_0)} u_0(a) \dd a + e^{-\mu_\infty t} \int_{a_0}^{+\infty} u_0(a) \dd a \nonumber \\
    &\qquad + c_0 \int_0^{t-a_0} e^{-\mu_\infty (t-a-a_0)}b(a) \dd a, \nonumber
\end{align}
from where it follows that
\begin{align*}
    e^{\mu_\infty t} I(t) &= \int_0^{a_0} e^{-\int_a^{2a_0 - a} \mu(l) \dd l} u_0(a) \dd a + \int_{a_0}^{+\infty} u_0(a) \dd a + c_0 \int_0^{t-a_0} e^{\mu_\infty(a+a_0)} b(a) \dd a.
\end{align*}
The above expression implies that $I$ is a solution of the differential equation
\[
    I'(t) = -\mu_\infty I(t) + e^{-\int_0^{a_0} \mu(l) \dd l} b(t-a_0), \quad t \geq a_0,
\]
with $I(a_0) = \int_0^{a_0} e^{-\int_a^{2a_0 - a} \mu(l) \dd l} u_0(a) \dd a + \int_{a_0}^{+\infty} u_0(a) \dd a$. Also, using \eqref{expr_I} we deduce for $t\geq a_0$ as well
\begin{align*}
    b(t) &= f\Bigg( \int_0^{a_0} \beta(a)e^{-\int_0^a \mu(l) \dd l} b(t-a) \dd a +  \beta_\infty \int_{a_0}^t e^{-\int_0^a \mu(l) \dd l} b(t-a) \dd a  \\
    &\qquad + \beta_\infty \int_t^{+\infty} e^{-\int_{a-t}^a \mu(l) \dd l} u_0(a-t) \dd a \Bigg) \\
    &= f\left( \int_0^{a_0} \beta(a)e^{-\int_0^a \mu(l) \dd l} b(t-a) \dd a + \beta_\infty I(t) \right).
\end{align*}

Therefore, we have proved that every solution of the age-structured model yields a solution of the system \eqref{delay_sys} for $t\geq a_0$, and therefore the dynamics of the Gurtin-MacCamy model relies on that of the delayed system. From this remark, the proof of Theorem \ref{th:main_2} follows the same reasoning as that one of Theorem \ref{th:main} using Theorem \ref{th:threshold_2}.

\section*{Acknowledgments}

F.H. acknowledges the support of ANID-Subdirección de Capital Humano/Doctorado Nacional/2024-21240616 and the support of the Région Normandie through the RIN50 project “DAMES”. We would like to express our gratitude to S. Trofimchuk for valuable discussions and for calling authors' attention to the work \cite{MR1822211}.

\vspace{5mm}

\noindent {\bf Data Availability} No datasets were generated or analysed during the current study.

\section*{Declarations}

{\bf Competing interests} The authors declare no competing interests.

\bibliography{biblio.bib}
\bibliographystyle{abbrv}

\appendix

\section{The strongly order-preserving property}

Let us consider the semiflow generated by the Gurtin-MacCamy model \eqref{GM-model}, $U(t)\colon \L^1_+(0,+\infty) \to \L^1_+(0,+\infty)$, which is given by the method of the characteristics; i.e., $U(t)u_0 = u(t,\cdot)$, where
\[
    u(t,a) = \begin{cases}
        e^{-\int_{a-t}^a \mu(l) \dd l} u_0(a-t), & a\geq t, \\
        e^{-\int_0^a \mu(l) \dd l} b(t-a), & a\leq t,
    \end{cases}
\]
with $b\colon \R_+ \to \R_+$ being the unique continuous solution of the integral equation
\[
     b(t) = f\left( \int_t^{+\infty} \beta(a) e^{-\int_{a-t}^a \mu(l) \dd l} u_0(a-t) \dd a + \int_0^t \beta(a)e^{-\int_0^a \mu(l) \dd l} b(t-a) \dd a \right).
\]
The proof of Theorem \ref{th:main} relies on the sharp threshold property for the semiflow generated by the solutions of the integral equation for $b$, that is, for the semiflow on a subset of $C_+^0([-a^*,0])$ as $(t,\varphi)\mapsto b_t^\varphi$, where $b_t$ denotes the history of $b$. In turn, this property is immediately deduced from the unorderedness of the separatrix, which follows from the generic asymptotic behavior for strongly monotone dynamical systems \cite{MR921986}. This description of the long-term dynamics can also be extended for semiflows defined over an ordered Banach space whose order is given by a non-solid cone, i.e., a cone with an empty interior, using the notion of strongly order-preserving semiflow \cite{MR1319817,MR1112067}. A semiflow $\Phi$ is called \emph{strongly order-preserving} if it is monotone and whenever $x<y$ there exist open subsets $U,V\subseteq X$ and $t_0>0$ such that $x\in U$, $y\in V$ and $\Phi(t_0,U) \leq \Phi(t_0,V)$.

Thus, Proposition \ref{compact_semifl} shows that the semiflow $(t,\varphi)\mapsto b_t^\varphi$ is eventually strongly monotone when $a^*<+\infty$, and therefore is strongly order-preserving. In contrast, the semiflow generated by the Gurtin-MacCamy model in $\L^1_+(0,+\infty)$ is never strongly order-preserving. In fact, consider any two initial distributions $u_0,v_0\in \L^1_+(0,+\infty)\cap \L^\infty_+(0,+\infty)$ such that $u_0 < v_0$, and denote by $\tilde{u}$ and $v$ the associated solutions to the age-structured model. In addition, let $U,V\subseteq \L^1_+(0,+\infty)$ be two open neighborhoods of $u_0$ and $v_0$, respectively. Now, set $\eta,\tilde{u}_0\in \L^1_+(0,+\infty)$ defined as $\eta = \frac{1}{\varepsilon} \1_{(0,\varepsilon^2)}$ and $\tilde{u}_0 = u_0 + \eta$, with $\varepsilon>0$. Taking $\varepsilon$ small enough, we assure that $\tilde{u}_0 \in U$ and $\tilde{u}_0(a) > v_0(a)$ a.e. in $(0,\varepsilon^2)$. Then, for each $t\geq 0$ it follows that
\[
    \tilde{u}(t,a) = e^{-\int_{a-t}^a \mu(l) \dd l} \tilde{u}_0(a-t) > e^{-\int_{a-t}^a \mu(l) \dd l} v_0(a-t)= v(t,a), \quad \text{a.e.}~ a\in(t,t+\varepsilon^2),
\]
and
\[
    \tilde{u}(t,a) = e^{-\int_{a-t}^a \mu(l) \dd l} u_0(a-t) \leq e^{-\int_{a-t}^a \mu(l) \dd l} v_0(a-t) = v(t,a), \quad \text{a.e.}~ a\in(t+\varepsilon^2,+\infty).
\]
This implies that $\Phi(t,\tilde{u}_0) = \tilde{u}(t,\cdot) \not \leq v(t,\cdot) = \Phi(t,v_0)$ for all $t\geq 0$, and hence the proof of the claim is finished.

On the other hand, when $a^*=+\infty$, the solution to the Volterra equation is not completely determined by the history over a compact interval, what naturally leads us to consider the space $BUC_+^0((-\infty,0])$ of bounded and uniformly continuous functions on the left half-axis. Nevertheless, the semiflow $(t,\varphi)\mapsto b_t^\varphi$ defined on a certain subspace $Y$ of $BUC_+^0((-\infty,0])$, mainly to avoid the discontinuity problem at zero, fails to be strongly order-preserving. Indeed, let us consider two initial data $\varphi,\psi\in Y$ such that $\varphi < \psi$ and $\varphi \equiv \psi$ on $(-\infty,\theta_0]$, for some $\theta_0<0$. Let $U,V\subseteq Y$ be two open neighborhoods of $\varphi$ and $\psi$, respectively. Consider a smooth function $\eta\colon (-\infty,0]\to \R$ such that $0\leq \eta\leq \varepsilon$ and
\[
    \eta(\theta) = \begin{cases}
        \varepsilon, & \theta\leq \theta_0 - 1, \\
        0, & \theta\in [\theta_0,0],
    \end{cases}
\]
with $\varepsilon>0$, and define $\tilde{\varphi} = \varphi + \eta$. Then, taking $\varepsilon$ small enough, we assure that $\tilde{\varphi}\in U$ and clearly $\tilde{\varphi} \not\leq \psi$. As the semiflow is defined in terms of the history of the solution over the whole left half-axis, the relation $b_t^{\tilde{\varphi}} \not\leq b_t^\psi$ holds for all $t\geq 0$, showing that the semiflow is not strongly order-preserving. Therefore, it is not possible to apply the theory of monotone dynamical systems using this particular formulation of the problem.

This problem is handled in Section \ref{sec:non_compact} by uncoupling the dynamics in an integro-differential system where one variable stores the information of the population for large ages, and therefore it defines a semiflow over a space of continuous functions defined on a compact interval. The generated smiflow turns out to be strongly order-preserving, and therefore the generic behavior of its orbits is known.
\end{document}